\documentclass[11pt, a4paper]{article}
\usepackage{amsmath}
\usepackage{amssymb}
\usepackage{amsfonts}
\usepackage[ansinew]{inputenc}
\usepackage{amsthm}
\usepackage{geometry}
\usepackage{hyperref}
\usepackage[UKenglish]{babel}

\author{Lennart Ronge}
\date{\today}
\title {Index Theory for Globally Hyperbolic Spacetimes}
\def\head#1{\textbf{#1}}
\def\limn{\lim_{n\rightarrow \infty}}
\def\<{\left<}
\def\>{\right>}
\def\Rp{{\mathbb{R}_{\geq0}}}
\def\C{{\mathbb{C}}}
\def\R{{\mathbb{R}}}
\def\N{{\mathbb{N}}}
\def\mat#1#2#3#4{{\begin{pmatrix} #1&#2\\ #3&#4 \end{pmatrix}}}

\def\famop{A(t)_{t\in [0,T]}}

\def\LH{L^2([0,T],H)} 
\def\norm#1{\left|\left|#1\right|\right|}
\def\ddt{\frac{d}{dt}}

\def\insum#1{\sum\limits_{#1=0}^\infty}
\def\pdt{\frac{\partial}{\partial t}}
\def\pds{\frac{\partial}{\partial s}}
\def\pd#1{\frac{\partial}{\partial #1}}

\newtheorem{df}{Definition}[section]
\newtheorem{lemma}[df]{Lemma}
\newtheorem{thm}[df]{Theorem}
\newtheorem*{thm*}{Theorem}

\DeclareMathOperator{\Dom}{Dom}
\DeclareMathOperator{\Ran}{Ran}
\DeclareMathOperator{\Ker}{Ker}
\DeclareMathOperator{\Coker}{Coker}
\DeclareMathOperator{\APS}{APS}
\DeclareMathOperator{\ind}{ind}
\DeclareMathOperator{\sfl}{sf}
\DeclareMathOperator{\spann}{span}
\DeclareMathOperator{\Dim}{Dim}
\DeclareMathOperator{\spec}{sp}

\begin{document}
\maketitle
\tableofcontents
\newpage
\section {Introduction}
In [B\"aSt], a Lorentzian version of the Atiyah-Patodi-Singer index theorem is derived for globally hyperbolic spacetimes. To this end, the authors first show the equality\footnote{In [B\"aSt], there is an additional summand on the right hand side. This summand disappears if the boundary conditions are chosen slightly differently, as will be done in this paper.}
$$\ind(D_{\APS})=\sfl(A),$$
where $D_{\APS}$ is related to the Dirac operator with Atiyah-Patodi-Singer boundary conditions and $\sfl(A)$ is the spectral flow of $A$, which is the family of Dirac operators on spacelike slices of a foliation of the spacetime. To show the equality above, B\"ar and Strohmaier investigate the evolution operator associated to $A$.\\
The purpose of this paper is to derive the aforementioned equality in a more general functional analytic setting:
\begin{itemize}
\item $\famop$ is a strongly continuously differentiable family of self-adjoint Fredholm operators on some Hilbert space $H$. All $A(t)$ have the same domain.
\item $D$ is the closure of the operator $\ddt-iA$ on $L^2([0,T],H)$. 
\item $D_{\APS}$ is the operator $D$ restricted to those functions $f$ with $f(0)$ in the range of the negative spectral projection of $A(0)$ and $f(T)$ in the range of the positive spectral projection of $A(T)$.
\end{itemize}
At first glance, this looks similar to the result of [RoSa], where the equality
$$\ind\left(\ddt-A\right)=\sfl(A)$$
is shown under similar conditions. However, there are significant differences to the case at hand. Firstly, in [RoSa] the operators are defined on the entire real line, rather than an interval with boundary conditions. Secondly, the additional $i$ leads to qualitatively different behavior of the operator $D$. In [RoSa], both the norm of $\ddt f$ and of $Af$ can be estimated by the graph norm $||f||_{\ddt-A}$. In the setting at hand, however, the equation
$$Df=0$$
has solutions with arbitrarily large $\ddt f$ and $Af$. In fact, the equation has a unique solution for any initial value. This can be used to define an evolution operator, which will play an important role in this paper.\\\\
In Section \ref{s2}, some background facts will be given and some straightforward statements about families of operators will be derived for later use. In Section \ref{s3}, the setting is described and the spectral flow is defined, following [PH]. The spectral flow is then related to the Fredholm \index of a pair of projections. Section \ref{s4} shows that the equations
$$Df=g$$
and
$$f(s)=x$$
have a unique solution for every $x\in H$,$s\in [0,T]$ and $g\in\LH$. The evolution operator will be a family $(Q(t,s))_{s,t\in [0,T]}$ of isometries defined in such a way that for $g=0$ the solution of the above is given by
$$f(t)=Q(t,s)x,$$
i.e. it describes evolution subject to the equation $Df=0$. The proofs in section \ref{s4} are mostly from [PA], reformulated for the setting at hand.
In chapter \ref{s5}, we will prove the main theorem of this thesis:
\begin{thm*}
\head{(Main Theorem)}
If $(D|_{[0,t]})_{\APS}$ is Fredholm for all $t\in[0,T]$, we have
$$\ind(D_{\APS})=\sfl(A).$$
\end{thm*}
Here $D|_{[0,t]}$ is the "restriction" of $D$ to the interval $[0,t]$. The theorem will be proven by showing that both sides of the equation are equal to the \index of a certain Fredholm pair of projections. This pair consists of the negative spectral projections at times $0$ and $T$ of the operator 
$$\hat A(t)=Q(0,t)A(t)Q(t,0).$$
The main theorem contains as a special case the equality derived in [B\"aSt]. The general idea of the proof and some of its parts are similar as in [B\"aSt], while other parts are different. In particular, the use of Fredholm pairs and the aforementioned projections allows for a much wider generalization than a straightforward adaption of the arguments of [B\"aSt] would.  Section \ref{s6} will describe a counterexample which shows that Fredholmness of $(D|_{[0,t]})_{\APS}$ is not a consequence of the other assumptions.

\section{Preliminaries}
\label{s2}
\subsection{General Notation}
All vector spaces used in this paper will be complex.
For linear operators, we will write
$$AB:=A\circ B$$
for composition and 
$$Ax:=A(x)$$
for evaluation. $I$ denotes the identity (it will be clear from the context on which space). For a linear operator $A$ and a scalar $\lambda$, write 
$$A+\lambda:=A+\lambda I.$$
The adjoint of $A$ will be denoted by $A^*$.

$B(X,Y)$ denotes the space of bounded operators between normed spaces $X$ and $Y$ and 
$$B(X):=B(X,X).$$

A projection will always mean an orthogonal projection, i.e. a self-adjoint idempotent operator. 

An unbounded operator on a Hilbert space will be a linear map from a dense subspace to the whole space. 

For integrals, the convention
$$\int\limits_s^tf(r)dr:=-\int\limits_t^sf(r)dr$$
for $s>t$ will be used. 

For any function $f$, the restriction of $f$ to some subset $X$ of its domain will be denoted by $f|_X$. In cases where the codomain is important, i.e. for Fredholm properties, $[f:X\rightarrow Y]$ will denote $f$ with domain restricted to $X$ and codomain restricted to $Y$.

\subsection{Notation of Norms}
Various norms will be used in this paper. If only one norm is applicable, this will simply be denoted by $||\cdot||$. Otherwise, $||\cdot||$ denotes the norm of the Hilbert space $H$ (to be introduced later) when applied to an element of $H$ and the operator norm with respect to the norm of $H$ when applied to an operator.
Similarly, $\<\cdot,\cdot\>$ will denote the scalar product on $H$ if more than one scalar product is applicable.
$||\cdot||_X$ will be used to denote the norm of some normed space $X$.  For an unbounded operator $U$ on a  Hilbert space, $||\cdot||_U$ denotes the graph norm on the domain $\Dom(U)$, which is induced by the graph scalar product
$$\<x,y\>_U=\<x,y\>+\<Ux,Uy\>.$$
This is chosen such that the map $x\mapsto (x,Ux)$ becomes an isometry from $\Dom(U)$ to the graph of $U$. If $U$ is closed (i.e. has a closed graph) this turns $\Dom(U)$ into a Hilbert space. If nothing else is specified, the graph norm is used as the norm on the domain of an unbounded operator.\\
 $||\cdot||_{X\rightarrow Y}$ denotes the operator norm as a map from $X$ to $Y$.\\
If the operator norm is to be taken with respect to non-canonical norms $||\cdot||_1$ and $||\cdot||_2$, it will be denoted
$$||\cdot||_{||\cdot||_1\rightarrow ||\cdot||_2}$$
For $f,g:[0,T]\rightarrow H$, $||f||_{L^p}:=||(||f(\cdot)||)||_{L^p}$ denotes the norm of $f$ in $L^p([0,T],H)$ and 
$$\<f,g\>_{L^2}:=\int\limits_0^T \<f(t),g(t)\>dt$$
denotes the scalar product of $\LH$.

\subsection{Functional Calculus}
Throughout this paper, it is assumed that the reader is familiar with the Spectral Theorem for unbounded self-adjoint operators, which allows to define to $f(A)$, for a self-adjoint operator $A$ and a Borel measurable function $f$ on its spectrum. One version of it can be found, for example, in [PE] (Theorem 5.3.8 and surroundings). This functional calculus is compatible with addition, multiplication and composition in the ways the notation suggests. Note that $\chi_a(A)$, with $\chi_a$ denoting the characteristic function of $\{a\}$, is the projection onto the eigenspace of $A$ with eigenvalue $a$. $f(A)$ commutes with all operators that commute with $A$. If $U$ is a unitary operator, $U^*AU$ is self-adjoint as well and $f(U^*AU)=U^*f(A)U$.

\subsection{Integration of Banach Space Valued Functions}
Let $J\subset \R$ be a compact interval and $X$ be a Banach space (Integration works for general measure spaces instead of $J$, but we will only need a compact interval).

A function $f:J\rightarrow X$ is called strongly measurable, if it is an almost everywhere uniform limit of countably-valued measurable functions, i.e. functions of the form
$$\insum{n}\chi_{E_n}x_n,$$
where the $E_n$ are pairwise disjoint measurable sets, the $x_n$ are elements of $X$ and the sum is taken pointwise (at every point, only one summand is non-zero). Since we will not need any other notion of measurability, we will write measurable for strongly measurable.

Continuous functions are measurable. For measurable functions whose norm is integrable, there is a notion of integration called the Bochner integral, which has many of the properties of the usual Lebesgue integral (see Chapter 3 of [HiPh] for further details). In particular, it coincides with the Riemann integral for continuous functions and the fundamental theorem of calculus holds. Integration can be interchanged with bounded operators, i.e.
$$\int\limits_J Sf(t)dt=S\int\limits_Jf(t)dt$$
for a bounded operator $S$. Taking $S$ to be the inclusion shows that an integral with respect to a stricter norm on a subspace agrees with the integral with respect to the weaker norm of the ambient space. 

The inequality
$$\norm{\int\limits_Jf(t)dt}\leq\int\limits_J\norm{f(t)}dt$$
holds, as well as Lebesgue's dominated convergence theorem.
\begin{lemma}
\label{diffint}
If $f:J\times J\rightarrow X$ is continuously differentiable, we have
$$\pdt\int\limits_s^t f(t,r) dr=f(t,t)+\int\limits_s^t \pdt f(t,r) dr.$$
\end{lemma}
\begin{proof}
For $h\in \R\backslash \{0\}$ such that $t+h\in J$, we have
\begin{align*}
&h^{-1}\left(\int\limits_s^{t+h}f(t+h,r)dr-\int\limits_s^t f(t,r)dr\right)\\
&=h^{-1}\int\limits_s^t f(t+h,r)-f(t,r)dr+h^{-1}\int\limits_t^{t+h} f(t+h,r)-f(t,r)dr+h^{-1}\int\limits_t^{t+h} f(t,r)dr.
\end{align*}
Letting $h$ tend to $0$, the first term converges to 
$$\int\limits_s^t \pdt f(t,r)dr$$
by dominated convergence, as  
$$||f(t+h,r)-f(t,r)||\leq h\max\limits_{x,y\in J}\norm{\pd{x} f(x,y)}.$$
The second term converges to 0 by the same inequality.
The last term converges to $f(t,t)$ by the fundamental theorem of calculus.
\end{proof}

One can define $L^p(J,X)$ in the usual way, as the space of measurable functions $f$ such that
$$||f||_{L^p(J,X)}:=\norm{(||f(\cdot)||)}_{L^p(J,\R)}$$
is finite, modulo functions vanishing outside a set of measure zero. This is a Banach space. For $p<q$, we have $L^q(J,X)\subseteq L^p(J,X)$ with bounded inclusion, as an immediate consequence of the corresponding result for real valued functions (using that $J$ is a compact interval). With the usual amount of sloppiness, functions will be identified with their $L^p$-classes and if a class has a continuous representative, it will be identified with that representative. It is sometimes convenient to consider a dense subspace of "nice" functions, like the one provided in the following lemma:
\begin{lemma}
\label{Sdense}
For $p\in[1,\infty)$, the set 
$$S(J,X):=\left\{\left.\sum\limits_{n=1}^{N}\phi_n x_n\right|N\in\N, \phi_n\in C^\infty(J,\R), x_n\in X\right\}$$
is dense in $L^p(J,X)$.
\end{lemma}
\begin{proof}
As any characteristic function can be approximated by $C^\infty$-functions in $L^p(J,\R)$, all finitely valued functions in $L^p(J,X)$ can be approximated by elements in $S(J,X)$. Finitely valued functions are dense in the countably-valued functions, as $\sum\limits_{n=0}^N\chi_{E_n}x_n$ converges in $L^p$ to $\insum{n}\chi_{E_n}x_n$ by dominated convergence. As uniform convergence almost everywhere implies $L^p$-convergence on a compact interval, countably valued functions are dense in $L^p(J,X)$ by definition of measurability. Putting everything together, we obtain that $S(J,X)$ is dense in $L^p(J,X)$.
\end{proof}

\subsection{Compact Operators and Fredholm Operators}
Let $X$ and $Y$ be Hilbert spaces. An operator is called compact if it maps the unit ball into some compact set. The set $K(X,Y)$ of compact operators from $X$ to $Y$ is a closed ideal in $B(X,Y)$. Operators of finite rank (dimension of the range) are dense in $K(X,Y)$. \\
An operator $F:X\rightarrow Y$ is called Fredholm if its kernel $\Ker(F)$ and its (vector space) cokernel 
$$\Coker(F)= Y/\Ran(F)$$
are finite dimensional. In this case, its Fredholm index is defined by
$$\ind(F)=\Dim(\Ker(F))-\Dim(\Coker(F)).$$
If $F$ is Fredholm, $F$ has closed range. If $F$ has closed range, 
$$\Coker(F)\cong \Ran(F)^\bot= \Ker(F^*).$$
If $F$ is a Fredholm operator and $K$ is a compact operator, $F+K$ is Fredholm and
$$\ind(F+K)=\ind(F).$$
If any of the two operators $F$,$G$ and $FG$ are Fredholm, then so is the third and we have
$$\ind(FG)=\ind(F)+\ind(G).$$
The adjoint of a compact operator is compact, the adjoint of a Fredholm operator is Fredholm.
An operator $F\in B(X,Y)$ is Fredholm if and only if it is "invertible up to compact operators", i.e. there is $G\in B(Y,X)$, such that $FG-I$ and $GF-I$ are compact. In this case, $G$ is called a parametrix for $F$. 
An unbounded operator $F$ on $X$ is called Fredholm, if it is Fredholm in $B(\Dom(F),X)$.

\subsection{Fredholm Pairs}
\label{sympair}
Let $P$ and $Q$ be projections in a Hilbert space (recall that "projection" means "orthogonal projection"). The pair $(P,Q)$ is called a Fredholm pair, if $[Q:{\Ran(P)}\rightarrow{\Ran(Q)}]$ is Fredholm. In this case the index of $(P,Q)$ is defined to be the Fredholm index of $[Q:{\Ran(P)}\rightarrow{\Ran(Q)}]$. We shall only need three theorems about Fredholm pairs. More information on them can be found in [AvSeSi].
\begin{lemma}
For two projections $P$ and $Q$, the following are equivalent
\begin{itemize}
\item $(P,Q)$ is a Fredholm pair
\item $(Q,P)$ is a Fredholm pair
\item $(I-P,I-Q)$ is a Fredholm pair
\end{itemize}
\end{lemma}
\begin{proof}
Proposition 3.1 in [AvSeSi] states the $(P,Q)$ is Fredholm if and only if 1 and -1 are not in the essential spectrum of $P-Q$. As this is true if and only if 1 and -1 are not in the essential spectrum of
$$(I-P)-(I-Q)=Q-P,$$
the claim follows.
\end{proof}
\begin{lemma}
\label{compfred}
If $P-Q$ is compact, then $(P,Q)$ is a Fredholm pair.
\end{lemma}
\begin{proof}
$[P:{\Ran(Q)}\rightarrow\Ran (P)]$ is a parametrix for $[Q:{\Ran(P)}\rightarrow\Ran (Q)]$, as
$$(QP-I)|_{\Ran(Q)}=Q(P-Q)|_{\Ran(Q)}$$
and
$$(PQ-I)|_{\Ran(P)}=P(Q-P)|_{\Ran(P)}$$
are compact.
\end{proof}
The last theorem is shown in [LE] (lemma 3.2). It states that continuous families of Fredholm pairs have constant index:
\begin{thm}
\label{contfred}
If $P,Q:[0,1]\rightarrow B(X)$ are continuous paths of projections in some Hilbert space $X$, such that $(P(t),Q(t))$ is a Fredholm pair for all $t\in[0,1]$. Then 
$$\ind(P(0),Q(0))=\ind(P(1),Q(1)).$$
\end{thm}

\subsection{Some Considerations on Families of Operators}
This subsection contains some theorems on strongly continuously differentiable families of operators between Banach spaces that are presented here so as to make the presentation of the main argument later on more straightforward. None of the results presented here are particularly surprising or new, the point is to check that things work as expected.\\\\
For this whole section, let $X$, $Y$ and $Z$ be Banach spaces, let $J$ be a compact interval and let $S:J\rightarrow B(X,Y)$ and $R:J\rightarrow B(Y,Z)$ be families of operators. \\\\
A family of operators $J\rightarrow B(X,Y)$ is called strongly continuous, if it is continuous with respect to the strong operator topology on $B(X,Y)$. It is called strongly continuously differentiable, if it is differentiable with respect to the strong operator topology and the derivative is strongly continuous.

We want to show that (strong) continuity and continuous differentiability of $R$ and $S$ imply the  same for $RS$. Note that this implies the analogous results for evaluation instead of composition: If $f:[0,T]\rightarrow Y$ is a function, we can set $S:[0,T]\rightarrow B(\C,Y)$, $S(t)(1)=f(t)$. Then $R(t)f(t)$ has the same regularity properties that $R(t)S(t)$ has in the strong topology. 

\begin{lemma}
\label{dbound}
If $S$ is strongly continuous, it is uniformly bounded.
\end{lemma}
\begin{proof}
As $S(t)x$ is continuous in $t$ for every $x$, $S$ is a pointwise bounded family in $B(X,Y)$. Thus, by the Banach-Steinhaus uniform boundedness principle, it is uniformly bounded.
\end{proof}

\begin{lemma}
If $S$ and $T$ are strongly continuous, the map $t\mapsto R(t)S(t)$ is strongly continuous on $J$.
\end{lemma}
\begin{proof}
Multiplication of operators is continuous when restricted to subsets on which the operator norm is bounded. As $R$ and $S$ have bounded range by the previous lemma, $t\mapsto R(t)S(t)$ can be viewed as a composition of continuous functions and hence is continuous again (with respect to the strong topology).
\end{proof} 
The following lemma asserts that the product rule holds for strong differentiation.
\begin{lemma}
\label{prodiff}
Let $Y_0\subseteq Y$ be a subspace of $Y$ with a stronger norm that turns it into a Banach space. Assume S is a strongly continuous family in $B(X,Y_0)$, strongly continuously differentiable in $B(X,Y)$. Assume $R$ is a strongly continuous family in $B(Y,Z)$ that restricts to a strongly differentiable family in $B(Y_0,Z)$. Then $t\mapsto R(t)S(t)$ is strongly continuously differentiable in $B(X,Z)$, with derivative $R'(t)S(t)+R(t)S'(t)$.
\end{lemma}
\head{Remark:} Note that this includes the special case $Y_0=Y$, in which the theorem asserts that compositions of strongly continuously differentiable functions are strongly continuously differentiable.
\begin{proof}
The calculation is the same as for any continuous bilinear map:
\begin{align*}
&\lim\limits_{h\rightarrow 0}h^{-1}(R(t+h)S(t+h)-R(t)S(t))\\
&=\lim\limits_{h\rightarrow 0}h^{-1}(R(t+h)-R(t))S(t+h)+R(t)h^{-1}(S(t+h)-S(t))
\end{align*}
Note that multiplication is not continuous in the strong topology, but it is continuous when restricted to bounded subsets. As the Banach-Steinhaus theorem guarantees uniform boundedness of the differential quotients (evaluated at some vector, the differential quotients converge, thus they are bounded) , this is sufficient to pull the limit into the products, viewing the first summand as a composition $X\rightarrow Y_0\rightarrow Z$ and the second as a composition $X\rightarrow Y\rightarrow Z$. We get that the above is equal to
\begin{align*}
&(\lim\limits_{h\rightarrow 0}h^{-1}(R(t+h)-R(t)))(\lim\limits_{h\rightarrow 0}S(t+h))+R(t)(\lim\limits_{h\rightarrow 0}h^{-1}(S(t+h)-S(t)))\\
&=R'(t)S(t)+R(t)S'(t)
\end{align*}
\end{proof}
\begin{lemma}
\label{densecont}
Assume $X_0\subseteq X$ is a dense subspace. If $S:[0,T]\rightarrow B(X,Y)$ is a uniformly bounded family of operators such that its restriction to $X_0$ is strongly continuous in $B(X_0,Y)$, then $S$ is strongly continuous.
\end{lemma}
\begin{proof}
This is a standard $\frac{\epsilon}{3}$-argument.
Let $C$ be an upper bound for $S$. Given $x\in X$, $s\in [0,T]$ and $\epsilon>0$, choose $x'\in X_0$ such that 
$$||x-x'||\leq \frac{\epsilon}{3C}.$$ 
Choose $\delta>0$ such that $|s-t|<\delta$ implies 
$$||S(t)x'-S(s)x'||\leq\frac{\epsilon}{3}.$$
Then for $|s-t|\leq\epsilon$ we have
$$||S(t)x-S(s)x||\leq||S(t)||||x-x'||+||S(s)||||x-x'||+||S(t)x'-S(s)x'||\leq\epsilon,$$
so $S(t)x$ is continuous around $t=s$.
\end{proof}

\begin{lemma}
\label{diffbound}
If $S$ is strongly differentiable, it is norm continuous.
\end{lemma}
\begin{proof}
Let $t\in J$ and let
$$F:=\left\{\left.\frac{1}{s-t}(S(s)-S(t))\right|s\in J \backslash\{t\}\right\}.$$
Since S is strongly differentiable, $(\frac{1}{s-t}(S(s)-S(t))x)$ is bounded for every $x$. By the Banach-Steinhaus uniform boundedness principle, $F$ is uniformly bounded by some constant $C$.
Thus 
$$||S(s)-S(t)||\leq C(s-t)\stackrel{s\rightarrow t}{\rightarrow}0.$$
\end{proof}

\begin{lemma}
\label{idiff}
If $S$ is strongly continuously differentiable and $S(t)$ is invertible for all $t\in J$, then the family
\begin{align*}
S^{-1}:J&\rightarrow B(X,Y)\\
t&\mapsto S(t)^{-1}
\end{align*}
is strongly continuously differentiable in B(Y,X) with derivative $-S^{-1}S'S^{-1}$.
\end{lemma}
\begin{proof}
By the previous theorem, $S$ is norm-continuous. As the inversion map is norm-continuous as well, $S^{-1}$ is norm-continuous (and in particular uniformly bounded).

Let $t\in J$ and $h\in\R$ small enough such that $t+h \in J$. Then for $y\in Y$
\begin{align*}
(S(t+h)^{-1}-S(t)^{-1})y&=S(t+h)^{-1}(S(t)-S(t+h))S(t)^{-1}y\\
&=-S(t+h)^{-1}(hS'(t)S^{-1}(t)y+o(h))\\
&=-hS(t)^{-1}S'(t)S(t)^{-1}y +o(h),
\end{align*}
which implies the claim.
\end{proof}

The point of the next theorem is to show that strong differentiability is preserved when changing domain and range, as long as the derivative permits it.
\begin{thm}
\label{strictdiff}
Assume $X_0\subseteq X$ is a dense subspace and $Y_0\subseteq Y$ is a subspace with a stricter norm, with respect to which it is complete. Assume additionally that $S(t)\in B(X,Y_0)$ for all $t\in J$  and that $t\mapsto S(t)|_{X_0}$ is  strongly differentiable in $B(X_0, Y)$ with a derivative extending to a continuous family $S'$ in $B(X,Y_0)$. Then $S$ is differentiable with derivative $S'$ in the strong topology of $B(X,Y_0)$.
\end{thm}
\begin{proof}
For $x\in X_0$, we have
$$S(t)x=S(t_0)x+\int\limits_{t_0}^tS'(r)xdr$$
with the integral being taken in $Y$. By the assumptions on $S'$, the integral also exists in $Y_0$ and, as the norm of $Y_0$ is stricter than that of $Y$, takes the same value. Thus we have the same equality in $Y_0$. Moreover, both sides are bounded with respect to the norm of $X$, so the equality holds for all $x\in X$. Thus $S'(t)x$ is the derivative of $S(t)x$ in $Y_0$ for all $x\in X$.
\end{proof}

\begin{lemma}
\label{compeq}
Let $X$, $Y$ be separable Hilbert spaces. For $K:X\rightarrow Y$, the following are equivalent:
\begin{enumerate}
\item $K$ is compact.
\item For any Hilbert space Z and any sequence of $R_n\in B(Y,Z)$ converging  to a limit $R\in B(Y,Z)$ in the strong topology, $R_nK$ converges to $RK$ in the operator norm.
\item If B is an (ordered) orthonormal Basis of $Y$ and $P_n$ denotes the  orthogonal projection onto the span of the first $n$ basis vectors (or all, if $\Dim(Y)<n$), $P_nK$ converges to $K$ in norm.
\item If B is an (ordered) orthonormal Basis of $X$ and $P_n$ denotes the  orthogonal projection onto the span of the first $n$ basis vectors (or all, if $\Dim(X)<n$), $KP_n$ converges to $K$ in norm.
\end{enumerate}
\end{lemma}
\head{Remark:} Separabity of Y can be enforced by replacing $Y$ with $K(X)$, which is separable if $X$ is separable. Thus all items apart from $3.$ are still equivalent, if $Y$ is not separable.
\begin{proof}
$"1\Rightarrow 2":$ By Banach-Steinhaus, the family $R_n$ is uniformly bounded by some constant $C$. If $K$ is compact, there is a sequence of finite rank operators converging to $K$ in the norm topology. For $\epsilon>0$, choose a finite rank operator $K'$ such that $$||K-K'||\leq\frac{\epsilon}{2(C+||R||)}.$$ As $\Ran(K')$ is finite dimensional, $R_n|_{\Ran(K')}$ converges to $R|_{\Ran(K')}$ in norm. Thus we can choose $n_0$ such that for all $n\geq n_0$, we have 
$$||(R-R_n)|_{\Ran(K')}||\leq\frac{\epsilon}{2||K'||}.$$
Then we get for all $n\geq n_0$:
$$||RK-R_nK||\leq||R(K-K')||+||R_n(K-K')||+||(R-R_n)K'||\leq\epsilon.$$
$"2\Rightarrow 3":$ This follows immediately, as $P_n$ converges strongly to the identity.\\
$"3\Rightarrow 1":$ $P_nK$ is of finite rank and hence compact. As the space of compact operators is norm closed, the limit $K$ is also compact.\\
$"1\Leftrightarrow 4":$ We have the following equivalences (using what is already shown):

$K$ is compact\\
$\Leftrightarrow$ $K^*$ is compact\\
$\Leftrightarrow$ $P_n K^*$ converges to $K^*$\\
$\Leftrightarrow$ $KP_n=(P_nK^*)^*$ converges to $K$.
\end{proof}

With this we can now show the following:

\begin{thm}
\label{dcomp}
If $S$ is strongly continuously differentiable and $S'(r)$ is compact for all $r\in J$, then $S(t)-S(s)$ is compact for all $s,t\in J$.
\end{thm}
\begin{proof}
Claim: $S'$ vanishes on the complement of a separable subspace.

Assume not. Then there is an uncountable orthonormal set $B$ with $$\sup\limits_{t\in J}||S'(t)b||>0$$ for all $b\in B$. As $(0,\infty)$ is the union of conutably many intervals of the form $[2^k,2^{k+1})$ for $k\in \mathbb{Z}$, the pigeon hole principle implies that there is $\epsilon>0$ such that
$$\sup\limits_{t\in J}||S'(t)b||>2\epsilon$$
for uncountably many $b$. By strong continuity of $S'$, there is a rational interval $I_b$ for each such $b$ such that $||S'(t)b||>\epsilon$ for $t\in I_b$. As there are only countably many rational intervals, we can apply the pigeon hole principle again to obtain an interval $I$ on which $||S'(t)b||>\epsilon$ holds for uncountably many $b$. Thus for $t\in I$, the operator $S'(t)$ is not the norm limit of finite rank operators (which vanish on almost all of the aforementioned $b$), contradicting compactness of $S'(t)$. 

This proves the claim. Thus we can assume without loss of generality that $X$ is separable, otherwise restrict to a separable subspace on whose complement $S'(t)$ vanishes.

Fix an orthonormal Basis $B$ of $X$.
By the previous lemma, $K\in B(X,Y)$ is compact if and only if $F_n(K):=||K(I-P_n)||$ converges to 0.
As
$$||(S(t)-S(s))(I-P_n)(x)||=\norm{\int\limits_s^tS'(r)(I-P_n)xdr}\leq\left|\int\limits_s^tF_n(S'(r))dr\right|\norm{x},$$
(with absolute values in case $t<s$) we have
$$F_n(S(t)-S(s))\leq \left|\int\limits_s^tF_n(S'(r))dr\right|.$$
By Banach Steinhaus, $S'(r)$ is uniformly bounded. By Lebesgue's dominated convergence theorem, we thus have
$$\limn \int\limits_s^tF_n(S'(r))dr=\int\limits_s^t\limn F_n(S'(r))dr=0.$$
Thus $F_n(S(t)-S(s))$ converges to 0, whence $S(t)-S(s)$ is compact.
\end{proof}

\section{Setting and Spectral Flow}
\label{s3}
\subsection{Setting}
\label{set}
\textbf{The objects defined in this section will be used throughout the paper. In particular, the Operator $A$ is always assumed to satisfy the following assumptions.}\\
Let $H$ be a Hilbert space. Let $W\subseteq H$ be a dense subset. Let $\famop$ be a family of unbounded self-adjoint Fredholm operators in $H$ with domain $W$. Use the graph norm of $A(0)$ as the norm of $W$. As $A(0)$ is self-adjoint and hence closed, $W$ is complete with respect to this norm. 
\begin{lemma}
All $A(t)$ are in $B(W,H)$
\end{lemma}
\begin{proof}
For $t\in[0,T]$, the graph of $A(t)$ is closed in $H\times H$. Thus it is also closed with respect to the stricter norm of $W\times H$. By the closed graph theorem, $A(t)$ is in $B(W,H)$.
\end{proof}
Additionally, assume that $t\mapsto A(t)$ is continuously differentiable with respect to the strong operator topology  on $B(W,H)$ ("strongly continuously differentiable").\\
The choice of norm for $W$ is somewhat arbitrary, but it doesn't matter, as the following lemma shows:
\begin{lemma}
\label{equinorm}
All $A(t)$ have uniformly equivalent graph norms. 
\end{lemma}
\begin{proof}
For any $s\in [0,T]$, $A(s)$ is self-adjoint, hence $A(s)+i$ is invertible. We have for $x\in W$
$$||x ||_{A(t)}=||(A(t)+i)x||\leq||(A(t)+i)(A(s)+i)^{-1}||||x||_{A(s)}$$
By lemmas \ref{dbound} and \ref{idiff}, $(A(t)+i)(A(s)+i)^{-1}$ is uniformly bounded, hence
$$||x ||_{A(t)}\leq C ||x ||_{A(s)},$$
with a constant $C$ independent of $s$ and $t$.
\end{proof}

For $s<t\in [0,T]$, let $D_0|_{[s,t]}$ denote the operator 
$$D_0|_{[s,t]}:C^1([s,t],W)\rightarrow L^2([s,t],H)$$
$$(D_0|_{[s,t]}f)(r):=f'(r)-iA(r)f(r).$$
\begin{lemma}
$D_0|_{[s,t]}$ is closable in $L^2([s,t],H)$.
\end{lemma}
\begin{proof}
For $f\in C^1([s,t],W)$ and $g\in C^1([s,t],W)$ with $g(s)=g(t)=0$, we have $$\<g,D_0|_{[s,t]}f\>_{L^2}=\<-D_0|_{[s,t]}g,f\>_{L^2}$$ by partial integration and self-adjointness of $A$. As the space of all such $g$ is dense, the adjoint of $D_0|_{[s,t]}$ is densely defined. This implies that $D_0|_{[s,t]}$ is closable.
\end{proof}
Let $D|_{[s,t]}$ denote the closure of $D_0|_{[s,t]}$, let $D:=D|_{[0,T]}$.

\subsection{The Spectral Flow}
Heuristically, the spectral flow of $A$ gives the number of eigenvalues of $A(t)$ (counted with multiplicities) crossing 0 as $t$ varies from $0$ to $T$, i.e. the number of negative eigenvalues becoming positive minus that of positive eigenvalues becoming negative. The approach and definition used are from [PH]. The results of this section are formulated for $A$, but everything works for any family of self-adjoint Fredholm operators that is norm continuous in $B(W,H)$. Note that $A$ is norm continuous by Lemma \ref{diffbound}.

\begin{df}
For an Interval $I\subset \R$ and $t\in[0,T]$, let
$$P_I(t):= \chi_{I}(A(t))$$
be the corresponding spectral Projection for A(t). Here $\chi_I$ denotes the characteristic function of $I$.
Let 
$$H_I(t):=\Ran(P_I)$$
be the corresponding spectral subspace.
For $a\in \R$, let
$$P_{<a}:=P_{(-\infty,a)}$$
\end{df}
We will need a theorem about how the spectrum and the spectral projections behave under continuous changes of the the operator. The following is an immediate consequence of Theorem 5.12 in chapter 6 of [KA]:
\begin{thm}
\label{kat}
For $t\in [0,T]$ and $a\in \R\backslash \spec(A(t))$, there is a neighborhood $U$ of $t$, such that for all $s\in U$ we have $a\notin \spec(A(s))$ and the spectral projection $P_{<a}(s)$ is norm continuous on $U$ (as a function of $s$).
\end{thm}
To define the spectral flow, we will need  finite dimensional spectral subspaces of the form $H_{[0,a)}$. To get those, we need to show that the spectrum of $A(t)$ is discrete in a neighborhood of $0$. We will show this first at every point $t$ and then for an open cover.
\begin{lemma}
\label{pointdisc}
For every $t\in [0,T]$, there is $a>0$ such that $\spec(A(t))\cap[-a,a]\subseteq \{0\}$. In particular $P_{[-a,a]}(t)$ has finite rank.
\end{lemma}
\begin{proof}
Some $x\in H$ is in  $\Ran(A(t))^\bot$ if and only if for all $y\in W$,
$$\<x,A(t)y\>=0.$$
By definition of the adjoint, this is true if and only if $x$ is in the domain of $A(t)^*$ and 
$$A(t)^*x=0.$$
Thus
$$\Ran(A(t))^\bot=\Ker(A(t)^*)=\Ker(A(t)).$$
As $A(t)$ is Fredholm, it has closed range, so we get that the restriction of $A(t)$ to
$$A(t)_r:=[A(t):W\cap \Ran(A(t))\rightarrow \Ran(A(t))]$$
produces a surjective map with kernel 0, i.e. an invertible map. As the spectrum of $A(t)_r$ is closed and does not contain 0, there is $a> 0$, such that 
$$\spec(A(t)_r)\cap [-a,a]=\emptyset$$
Using the splitting $H=\Ran(A)\oplus \Ker(A)$, we obtain that for non-zero $|\lambda|\leq a$, the operator
$$A(t)+\lambda=(A(t)_r+\lambda)\oplus\lambda id_{\Ker(A(t))}$$
is invertible, so $\lambda\notin \spec(A(t))$, showing the first claim. For the second claim, note that 
$$P_{[-a,a]}(t)=P_{\{0\}}(t)$$
is the projection onto $\Ker(A(T))$, which is finite dimensional since $A(t)$ is Fredholm.
\end{proof}
\begin{lemma}
\label{constrank}
If $J\subseteq \R$ is an interval and $P:J\rightarrow B(H)$ is a continuous family of projections, the rank of $P$ is constant as a map from $J$ into $\N\cup\{\infty\}$.
\end{lemma}
\begin{proof}
Let $s,t\in J$ be close enough such that $||P(s)-P(t)||<1$.
Assume for contradiction that $P(s)$ and $P(t)$ do not have the same rank. Without loss of generality, assume that $P(s)$ has the smaller rank.
Then $P(t)(\Ran(P(s)))$is a proper subspace of $(\Ran(P(t))$. It is finite dimensional and hence closed, so its complement is non-zero.
Thus there is a non-zero $x\in \Ran(P(t))$ such that for all $y\in \Ran(P(s))$, we have
$$0=\<x,P(t)y\>=\<P(t)x,y\>=\<x,y\>.$$
This means 
$$x\in \Ran(P(s))^\bot = \Ker(P(s)),$$
whence
$$||P(s)x-P(t)x||=||x||,$$
contradicting $||P(s)-P(t)||<1$.
Thus $P(s)$ and $P(t)$ have the same rank for all $s$ in some neighborhood of any $t\in [0,T]$, i.e the rank of $P$ is locally constant. As $J$ is connected, the rank of $P$ is constant.
\end{proof}

\begin{lemma}
\label{Udisc}
For every $t\in[0,T]$, there is a neighborhood $U$ of $t$ and $a>0$ such that $P_{[-a,a]}(s)$ has finite rank and $\pm a\notin \spec(A(s))$ for all $s\in U$.
\end{lemma}
\begin{proof}
Choose $a$ as in Lemma \ref{pointdisc}. Then $P_{[-a,a]}(t)$ has finite rank and $\pm a\notin \spec(A(t))$. By Theorem \ref{kat}, there is a neighborhood $U$ of $t$ such that $\pm a\notin \spec(A(s))$ for $s\in U$ and
$$P_{[-a,a]}=P_{<a}-P_{<-a}$$ is continuous on $U$. By Lemma \ref{constrank}, this implies that $P_{[-a,a]}$ has constant, and thus finite, rank on $U$.
\end{proof}

For defining the spectral flow, we will need suitable partitions.
\begin{df}
A partition 
$$0=t_0<t_1<...<t_N=T$$
with numbers $a_n\in \Rp$ for $1\leq n\leq N$ will be called a flow partition (for $A$), if  
\item for all $n\leq N$ and $t\in [t_{n-1},t_n]$ we have $a_n\notin \spec(A(t))$ and $H_{[0,a_n)}(t)$ is finite dimensional.
\end{df}

\begin{thm}
\label{flowpart}
There is a flow partition for $A$. This can be chosen in such a way that it is also a flow partition for $-A$.
\end{thm}
\begin{proof}
For every $t\in [0,T]$, choose a neighborhood $U_t$ with $b_t>0$, such that $P_{[-b_t,b_t]}(s)$ has finite range and $\pm b_t\notin \spec(A(s))$ for all $s\in U_t$ (using Lemma \ref{Udisc}).
By Lebesgue's number lemma, there is $N\in\N$ such that every interval of length less that $\frac{1}{N}$ is contained in one of the $U_t$. Let $t_n:=\frac{nT}{N}$.
Let $a_n:=b_t$ for some $t$ such that $[t_{n-1},t_n]\subseteq U_t$. Then the $t_n$ and $a_n$ form a flow partition. As we also have $-a_n\notin \spec(A(t))$ for $t\in[t_{n-1},t_n]$,  the $t_n$ and $a_n$ also form a flow partition for $-A$.
\end{proof}

\begin{df}
If $(t_n)_{0\leq n\leq N},(a_n)_{1\leq n\leq N}$ is a flow partition for $A(t)$, the spectral flow is defined as 
$$\sfl(A)=\sum\limits_{n=1}^N \Dim(H_{[0,a_n)}(t_n))-\Dim(H_{[0,a_n)}(t_{n-1}))$$
\end{df}
The intuition behind this is that, since $a_n$ is a spectral gap, all eigenvalues entering or leaving $[0,a_n)$ must do so through 0. The following theorem shows that the spectral flow is well-defined.
\begin{thm}
The spectral flow is independent of the choice of flow partition.
\end{thm}
\begin{proof}
For a flow partition $(t_n),(a_n)$, define $\alpha(t):=a_n$ for $t\in[t_{n-1}, t_n]$. As
\begin{align*}
&\Dim(H_{[0,a_n)}(t_n))-\Dim(H_{[0,a_n)}(t_{n-1}))\\
&=\Dim(H_{[0,a_n)}(t_n))-\Dim(H_{[0,a_n)}(\tau))+\Dim(H_{[0,a_n)}(\tau))-\Dim(H_{[0,a_n)}(t_{n-1})),
\end{align*}
inserting an additional point $\tau$ in the flow partition while choosing the $a_n$ such that $\alpha$ remains unchanged does not change the spectral flow. Thus any two flow partitions can be refined to have the same $(t_n)$ without changing their spectral flow.

Assume $(t_n),(a'_n)$ is another partition with the same $(t_n)$ as the first. Let $n\leq N$ and assume w.l.o.g. $a'_n\geq a_n$ (otherwise interchange $a_n$ and $a'_n$ in the following argument). The projection
$$P_{[a_n,a'_n)}=P_{<a'_n}-P_{<a_n}$$
has constant rank on $[t_{n-1},t_n]$ by Theorem \ref{kat} and Lemma \ref{constrank}. Thus we have
\begin{align*}
&\Dim(H_{[0,a'_n)}(t_n))-\Dim(H_{[0,a'_n)}(t_{n-1}))\\
&=\Dim(H_{[0,a_n)}(t_n))+\Dim(H_{[a_n,a'_n)}(t_n))-\Dim(H_{[0,a_n)}(t_{n-1}))-\Dim(H_{[a_n,a'_n)}(t_{n-1}))\\
&=\Dim(H_{[0,a_n)}(t_n))-\Dim(H_{[0,a_n)}(t_{n-1})).
\end{align*}
As this equality holds for all $n$, changing the $a_n$ does not change the spectral flow either, so it is the same for all flow partitions.
\end{proof}
From the definition it is obvious that the spectral flow is additive in the sense that for $r<s<t\in [0,T]$, we have
$$\sfl(A|_{[r,t]})=\sfl(A|_{[r,s]})+\sfl(A|_{[s,t]}).$$
The spectral flow is antisymmetric under sign change, up to counting the kernel twice:
\begin{lemma}
\label{minflow}
If $(t_n)$ is a partition with $a_n\in \Rp$ such that $-a_n\notin \spec(A(t))$ for all $t\in[t_{n-1},t_n]$ (in other words, $(t_n)$,$(a_n)$ is a flow partition for $-A$), then the spectral flow of $A$ is
$$\sfl(A)=\sum\limits_{n=1}^N \Dim(H_{[-a_n,0)}(t_{n-1}))-\Dim(H_{[-a_n,0)}(t_{n})),$$
or equivalently
$$\sfl(A)=-\sfl(-A)+\Dim\Ker(A(T))-\Dim\Ker(A(0)).$$
\end{lemma}
\begin{proof}
Choose a flow partition $(t_n),(a_n)$ for $A$ that is also a flow partition for $-A$ (using Theorem \ref{flowpart}). The projection
$$P_{[-a_n,a_n)}$$
has constant rank on $[t_{n-1},t_n]$ by Theorem $\ref{kat}$ and Lemma $\ref{constrank}$ (as above).
Thus
\begin{align*}
\sfl(A)&=\sum\limits_{n=1}^N \Dim(H_{[0,a_n)}(t_n))-\Dim(H_{[0,a_n)}(t_{n-1}))\\
&=\sum\limits_{n=1}^N \left(\Dim(H_{[-a_n,a_n)}(t_n))-\Dim(H_{[-a_n,0)}(t_{n}))\right.\\
&\ \ \ \ \left.-\Dim(H_{[-a_n,a_n)}(t_{n-1}))+\Dim(H_{[-a_n,0)}(t_{n-1}))\right)\\
&=\sum\limits_{n=1}^N -(\Dim(H_{[-a_n,0)}(t_{n}))-\Dim(H_{[-a_n,0)}(t_{n-1}))).
\end{align*}
As the spectral subspace of $A$ corresponding to an interval $I$ is that of $-A$ corresponding to $-I$, we have
\begin{align*}
&\sum\limits_{n=1}^N -(\Dim(H_{[-a_n,0)}(t_{n}))-\Dim(H_{[-a_n,0)}(t_{n-1})))\\
&=-\sfl(-A)+\sum\limits_{n=1}^N \Dim(H_{\{0\}}(t_{n}))-\Dim(H_{\{0\}}(t_{n-1}))\\
&=-\sfl(-A)+\Dim\Ker(A(T))-\Dim\Ker(A(0)).
\end{align*}
As the last line is independent of the choice of flow partition and the second set of equalities holds for any flow partition of $-A$, the theorem holds for other flow partitions of $-A$ as well.
\end{proof}
The spectral flow is unchanged by conjugating with unitaries:

\begin{thm}
\label{conflow}
Let $(U(t))_{t\in[0,T]}$ be a family of unitary operators such that spectral flow is defined for $t\mapsto U(t)^*A(t)U(t)$. Then 
$\sfl(U^*AU)=\sfl(A)$
\end{thm}
\begin{proof} 
$$\chi_{[0,a)}(U^*(t)A(t)U(t))=U^*(t)P_{[0,a)}(t)U(t)$$
is the projection onto $U^*(t)H_{[0,a)}(t)$, which has the same dimension as $H_{[0,a)}(t)$. Thus all summands in the definition of the spectral flow are unchanged by the conjugation, so the spectral flow is the same.
\end{proof}
\head{Remark:} The condition that $\sfl(U^*AU)$ is defined is purely formal. The defining formula for the spectral flow will always be well defined, because it takes the same values as for $A$.

The following theorem will be the first step toward proving the main theorem. Its proof uses arguments from section 4.2 of [B\"aSt] reformulated in terms of spectral projections and combines them with theorems \ref{kat} and \ref{contfred}.
\begin{thm}
\label{flowind}
If $(P_{<0}(0),P_{<0}(t))$ is a Fredholm pair for all $t\in[0,T]$, we have
$$\sfl(A)=\ind(P_{<0}(0),P_{<0}(T)).$$
\end{thm}

\begin{proof}
Let $(t_n)$, $(a_n)$ be a flow partition for $A$. 

The idea of the proof is to consider the pair $(P_{<0}(0),P_{<a_n}(t))$ for $t\in[t_{n-1},t_n]$, and show that it is Fredholm with constant index in $t$ on $[t_{n-1},t_n]$. Then we will compare its index to that of the pair $(P_{<0}(0),P_{<0}(t))$.  As the index of the former pair is constant on $[t_{n-1},t_n]$, changes in the latter index, as $t$ varies over $[t_{n-1},t_n]$, are the same as changes in the difference of the two indices. These will correspond exactly to the spectral flow over that interval.

For any projection $P$ let $P_r$ denote the restriction 
$$P_r:=[P:{\Ran(P_{<0}(0))}\rightarrow{\Ran(P)}]\ .$$
By the definition of Fredholm pairs and their index, the pair $(P_{<0}(0),P)$ is Fredholm with index k if and only if $P_r$ is Fredholm with index k.

Fix $n\leq N$. For $t\in[t_{n-1},t_n]$, define   
$$P_t:=[P_{<0}(t):{\Ran(P_{<a_n}(t))}\rightarrow{\Ran(P_{<0}(t))}]=[P_{<0}(t):{H_{(-\infty,a_n)}(t)}\rightarrow{H_{(-\infty,0)}(t)}].$$
 As $H_{(-\infty,0)}(t)$ is a subspace of $H_{(-\infty,a_n)}(t)$ with complement $H_{[0,a_n)}(t)$, this is Fredholm with index
$$\ind (P_t)=\Dim(\Ker(P_t))=\Dim(H_{[0,a_n)}(t)).$$
As projecting onto $H_{(-\infty,a_n)}(t)$ and then onto $H_{(-\infty,0)}(t)$ is the same as projecting onto $H_{(-\infty,0)}(t)$ immediately, we have
$$P_{<0}(t)_r=P_tP_{<a_n}(t)_r.$$
As the first two projections in the equation are Fredholm,
$P_{<a_n}(t)_r$ is Fredholm as well, and we have
$$\ind(P_{<0}(t)_r)=\ind(P_t)+\ind(P_{<a_n}(t)_r)=\Dim(H_{[0,a_n)}(t))+\ind(P_{<a_n}(t)_r).$$

By Theorem \ref{kat}, $P_{<a_n}(t)_r$ is continuous in $t$ on $[t_{n-1},t_n]$. By Theorem \ref{contfred}, $(P_{<0}(0), P_{<a_n}(t))$ has constant index for $t\in [t_{n-1},t_n]$. Thus we have
\begin{align*}
\ind(P_{<a_n}(t_n)_r)=\ind(P_{<0}(0),P_{<a_n}(t_n))=\ind(P_{<0}(0),P_{<a_n}(t_{n-1}))=\ind(P_{<a_n}(t_{n-1})_r)
\end{align*}
Moreover, as $P_{<0}(0)_r$ is the identity on $H_{(-\infty,0)}(0)$, it has index 0.

Combining everything, we get:
\begin{align*}
&\ind(P_{<0}(0),P_{<0}(T))\\
&=\ind(P_{<0}(T)_r)\\
&=\ind(P_{<0}(T)_r)-\ind(P_{<0}(0)_r)\\
&=\sum\limits_{n=1}^N \ind( P_{<0}(t_n)_r)-\ind(P_{<0}(t_{n-1})_r)\\
&=\sum\limits_{n=1}^N \Dim(H_{[0,a_n)}(t_n))+\ind(P_{<a_n}(t_n)_r)-\Dim(H_{[0,a_n)}(t_{n-1}))-\ind(P_{<a_n}(t_{n-1})_r)\\
&=\sum\limits_{n=1}^N \Dim(H_{[0,a_n)}(t_n))-\Dim(H_{[0,a_n)}(t_{n-1}))\\
&=\sfl(A).
\end{align*}
\end{proof}

\section{The Evolution Operator and the Cauchy Property}
\label{s4}
Recall the setting from section \ref{set}.
The purpose of this section is to show that for given $g\in L^2([0,T],H)$, $s\in [0,T]$ and $x\in H$, the equations
$$Df=g$$
$$f(s)=x\ $$
have a unique solution $f$. This will be called the Cauchy property. It basically follows from chapter 5 in [PA], but since the results needed here are slightly different from those derived in [PA] and some things are easier in this setting than in the one of [PA], a slightly adapted version of the arguments in [PA] will be presented below.
\subsection{The Evolution Operator}
The family $Q$ constructed in the following theorem will be called the evolution operator. It will be used throughout the rest of the paper. Its construction will occupy this subsection.
\begin{thm}
\label{evop}
There is a family of operators $Q(t,s):H\rightarrow H$ for $s,t\in[0,T]$, satisfying the following conditions (for all $s,t,r\in [0,T]$):
\begin{enumerate}
\item $Q(t,t)=I$ \label{a}
\item $Q(t,s)Q(s,r)=Q(t,r)$\label{b}
\item $Q(t,s)$ is an isometry (of $H$). \label{c}
\item $Q(t,s)(W)\subseteq W$ and $Q(t,s):W\rightarrow W$ is bounded.
\item $Q$ is strongly continuously differentiable in $B(W,H)$ with derivatives 
$$\pdt Q(t,s)=iA(t)Q(t,s)$$ and 
$$\pds Q(t,s)=-Q(t,s)iA(s).$$
\item $Q(t,s)x$ (as a function of $s$ and $t$) is continuous in $H$ for $x\in H$ and continuous in W for $x\in W$.
\end{enumerate}
\end{thm}
\head{Remark:}The function $f(t)=Q(t,s)x$ solves the equation Df=0 for $x\in W$ and thus, as $D$ is closed and $W$ is dense, for all $x\in H$. Applying $Q(t,s)$ can thus be thought of evolving from $s$ to $t$ subject to the equation $Df=0$. One should keep in mind that $Q$ is strongly continuous, but generally neither strongly differentiable nor norm continuous in $B(H)$ or $B(W)$, only in $B(W,H)$. Heuristically, its behavior can be thought of as rapid oscillation, with unbounded oscillation speed given by $A$.
\begin{proof}
For a constant family $A(t)=A_0$, the family $exp((t-s)iA_0)$ has all the above properties. The idea for constructing $Q$ in general is to piece together exponentials for piecewise constant approximations of $A$ and then take the limit for an ever finer partition.
\begin{lemma}
The family $Q(t,s)=exp((t-s)iA_0)$ has the properties required in Theorem \ref{evop} for the constant family $A(t)=A_0$. Moreover, it is an isometry of $W$ with respect to $||\cdot||_{A_0}$.
\end{lemma}
\begin{proof}
By a version of the Spectral Theorem, there is a finite measure space $(M,\mu)$ with a measurable function $\alpha:M\rightarrow \R$ and a unitary map $U:L^2(M)\rightarrow H$ such that
$$U^*A_0Uf=\alpha f,$$
for 
$$f\in U^*(\Dom(A))=\Dom(U^*AU)=\{f\in L^2(M)|\alpha f\in L^2(M)\}$$
i.e. $A_0$ is unitarily equivalent to a multiplication operator. The exponential function then takes the form
$$U^*exp(i(t-s)A_0)Uf=exp(i(t-s)\alpha)f$$
for $f\in L^2(M)$.

As the properties of Theorem \ref{evop} are preserved under conjugating with $U$, when replacing $A(t)$ with $U^*A_0U$, $W$ with $\Dom(U^*A_0U)$ and $H$ with $L^2(M)$, it suffices to show the properties for multiplication with $exp(i(t-s)\alpha)$ with those replacements.

As multiplication with $exp(i(t-s)\alpha)$ has properties 1. and 2., so does $exp(i(t-s)A_0)$. As the $L^2$-norm of a function only depends on its absolute value and we have for any $r\in M$
$$|exp(i(t-s)\alpha(r))f(r)|=|f(r)|$$
and
$$|\alpha(r) exp(i(t-s)\alpha(r))f(r)|=|\alpha(r)f(r)|,$$
we know that $U^*exp(i(t-s)A_0)U$ is an isometry of $L^2(M)$ and $\Dom(U^*AU)$. We can conclude that $exp(i(t-s)A_0)$ is an isometry of $H$ and $W$.

This leaves us with strong continuity and differentiability to show. Note that it suffices to show these for $t$, as changes in $s$ have the same effect as changes in $t$ with opposite sign. We'll start with continuity.

Fix $f\in L^2(M)$ and $\epsilon>0$. Let $b>0$ be large enough such that 
$$\int\limits_{|\alpha|^{-1}((b,\infty))}|2f|^2d\mu<\epsilon.$$
For $r\in\alpha^{-1}([-b,b])$,
$$|exp(i(t+h-s)\alpha(r))-exp(i(t-s)\alpha(r))|=\left|\int\limits_0^h i\alpha(r)exp(i(t+u-s)\alpha(r))du\right|\leq hb$$
converges to 0 uniformly for $h\rightarrow 0$. Thus 
\begin{align*}
&\limsup\limits_{h\rightarrow 0}||(exp(i(t+h-s)\alpha)-exp(i(t-s)\alpha))f||_{L^2(M)}^2\\
&\leq \limsup\limits_{h\rightarrow 0}\sup\limits_{r\in\alpha^{-1}([-b,b])} |exp(i(t+h-s)\alpha(r))-exp(i(t-s)\alpha(r))|^2\int\limits_{\alpha^{-1}([-b,b])}|f|^2d\mu\\
&+\int\limits_{|\alpha|^{-1}((b,\infty))}|(exp(i(t+h-s)\alpha)-exp(i(t-s)\alpha))f|^2d\mu\\
&\leq \int\limits_{|\alpha|^{-1}((b,\infty))}|2f|^2d\mu\\
&<\epsilon.
\end{align*}
As $\epsilon$ was arbitrary, $exp(i(t+h-s)\alpha)f$ converges to $exp(i(t-s)\alpha)f$ in $L^2(M)$. If $\alpha f$ is square integrable (i.e. $f\in \Dom(U^*A_0U)$), the same holds for $\alpha f$ instead of $f$. Thus multiplication with $exp(i(t-s)\alpha)$ is strongly continuous on $L^2(M)$ and $\Dom(U^*A_0U)$. We can conclude that $exp(i(t-s)A_0)$ is strongly continuous on $H$ and $W$.

For strong differentiability, pick $f$ such that $\alpha f\in L^2(M)$. For $\epsilon>0$, choose $\delta>0$ such that for $|u|<\delta$, we have
$$||(exp(i(t+u-s)\alpha)-exp(i(t-s)\alpha))\alpha f||_{L^2(M)}<\epsilon.$$
This is possible by the strong continuity shown in the previous paragraph. For $|h|<\delta$, we have
\begin{align*}
&||exp(i(t+h-s)\alpha)f-exp(i(t-s)\alpha)f-hi\alpha exp(i(t-s)\alpha)f||_{L^2(M)}\\
&=\left(\int\limits_M\left|\int\limits_0^h i\alpha(r)(exp(i(t+u-s)\alpha(r))-exp(i(t-s)\alpha(r)))f(r)du\right|^2d\mu (r)\right)^\frac{1}{2}.
\end{align*}
Using Minkowski's integral inequality, we can estimate this by
\begin{align*}
&\left| \int\limits_0^h ||i\alpha(exp(i(t+u-s)\alpha)-exp(i(t-s)\alpha))f||_{L^2(M)}du\right |\\
&\leq\epsilon|h|.
\end{align*}
Thus  $exp(i(t-s)\alpha)f$ is differentiable in $L^2(H)$ for all $f\in \Dom(U^*A_0U)$ with derivatives 
$$\pdt exp(i(t-s)\alpha)f=-\pds exp(i(t-s)\alpha)f=i\alpha exp(i(t-s)\alpha)f=exp(i(t-s)\alpha)i\alpha f.$$
We can conclude that $exp(i(t-s)A_0)$ is strongly differentiable in $B(W,H)$ with the desired derivatives (which are strongly continuous).
Thus we have all the desired properties.
\end{proof}

For $k\leq n\in \N$ let 
$$t_k^n:=\frac{kT}{n}$$
$$A_n(t):=A(t_k^n),\ for\ t\in [t_{k-1}^n,t_k^n]$$
For $s,t \in [t_{k-1}^n,t_k^n]$ define
$$Q_n(t,s):=exp((t-s)iA(t_k^n)).$$
For $s\in [t_{k-1}^n,t_k^n]$, $t\in [t_{l-1}^n,t_l^n]$ and $k<l$, define
$$Q_n(t,s):=Q_n(t,t_{l-1}^n)\left(\prod\limits_{j=k+1}^{l-1} Q_n(t_j^n,t_{j-1}^n)\right)Q_n(t_k^n,s).$$
The product is to be ordered right to left, i.e. $j=k+1$ corresponds to the rightmost factor. For $k>l$, define 
$$Q_n(t,s):=Q_n(t,t_l^n)\left(\prod\limits_{j=l+1 }^{k-1} Q_n(t_{j-1}^n,t_{j}^n)\right)Q_n(t_{k-1}^n,s),$$ 
with the product being ordered left to right this time. Note that these products are constructed in such a way that 
$$Q_n(s,s)=I$$
and
$$Q_n(t,r)Q_n(r,s)=Q_n(t,s).$$
 
As the definitions of $Q_n(t_i,s)$ with $l=i$ and $l=i+1$ agree, i.e. the families we patch together agree on interval boundaries, $Q_n(t,s)$ is strongly continuous in $t$. The same is true for $s$. Moreover, if $t$ and $s$ do not coincide with some $t_k^n$, we have
$$\pdt Q_n(t,s)=iA_n(t)Q_n(t,s)$$
and
$$\pds Q_n(t,s)=-Q_n(t,s)iA_n(t).$$
Every $Q_n(t,s)$ is an isometry $H\rightarrow H$ and maps $W$ to $W$, as this is true for all the exponentials from which it is composed.

\begin{lemma}
$||Q_n(s,t)||_{W\rightarrow W}$ is bounded by a constant independent of $s$,$t$ and $n$.
\label{Qbound}
\end{lemma}
\begin{proof} 

For $x\in W$ with $||x||=1$, we have
\begin{align*}
&\frac{||x||_{A(t_k^n)}^2}{||x||_{A(t_{k-1}^n)}^2}\\
&=\frac{1+||A(t_k^n)x||^2}{1+||A(t_{k-1}^n)x||^2}\\
&=1+\frac{||A(t_k^n)x||^2-||A(t_{k-1}^n)x||^2}{1+||A(t_{k-1}^n)x||^2}\\
&\leq 1+|||A(t_k^n)x||^2-||A(t_{k-1}^n)x||^2|\\
&\leq 1+\left|\ \int\limits_{t_{k-1}^n}^{t_k^n}\frac{d}{dr}||A(r)x||^2dr\right|\\
&\leq 1+\int\limits_{t_{k-1}^n}^{t_k^n}2||A(r)||_{W\rightarrow H}||A'(r)||_{W\rightarrow H}dr\\
&=:1+\Delta_k^n,\\
\end{align*}
where $\Delta_k^n$ is defined by the last equality.
For $s,t\in [t_{k-1}^n,t_k^n]$, $Q_n(t,s)$ is an isometry with respect to $||\cdot||_{A(t_k^n)}$.
From this we obtain for $s\in [t_{k-1}^n,t_k^n]$, $t\in [t_{l-1}^n,t_l^n]$, $k<l$ (the cases $k=l$ and $k>l$ work in the same way):
\begin{align*}
&||Q_n(t,s)||_{||\cdot||_{A(t_k^n)}\rightarrow ||\cdot||_{A(t_l^n)}}\\
&\leq||Q_n(t,t_{l}^n)||_{||\cdot||_{A(t_{l}^n)}\rightarrow ||\cdot||_{A(t_l^n)}}\left(\prod\limits_{j=k+1}^{l}||Q_n(t_j^n,t_{j-1}^n)||_{||\cdot||_{A(t_j^n)}}||id||_{||\cdot||_{A(t_{j-1}^n)}\rightarrow ||\cdot||_{A(t_j^n)}}\right)\\
& \cdot||Q_n(t_k^n,s)||_{||\cdot||_{A(t_{k}^n)}\rightarrow ||\cdot||_{A(t_k^n)}}\\
&=\prod\limits_{j=k+1}^{l-1}||id||_{||\cdot||_{A(t_{j-1}^n)}\rightarrow ||\cdot||_{A(t_j^n)}}\\
&=\prod\limits_{j=k+1}^{l-1}(1+\Delta_j^n)^\frac{1}{2}.
\end{align*}
 Write
$$\Delta:=\sum\limits_{j=0}^n\Delta_j^n=\int\limits_{0}^{T}2||A(r)||_{W\rightarrow H}||A'(r)||_{W\rightarrow H}dr.$$
The inequality of arithmetic and geometric mean gives
$$\prod\limits_{j=k+1}^{l-1}(1+\Delta_j^n)^\frac{1}{l-k-1}\leq\frac{1}{l-k-1}\sum\limits_{j=k+1}^{l-1}(1+\Delta_j^n).$$
Thus we can estimate 
\begin{align*}
\prod\limits_{j=k+1}^{l-1}(1+\Delta_j^n)^\frac{1}{2}&\leq \left(\frac{1}{l-k-1}\sum\limits_{j=k+1}^{l-1}(1+\Delta_j^n)\right)^\frac{l-k-1}{2}\\
&\leq\left(1+\frac{\Delta}{l-k-1}\right)^\frac{l-k-1}{2}\\
&=\left(1+\frac{\Delta}{l-k-1}\right)^{\frac{l-k-1}{\Delta}\frac{\Delta}{2}}\\
&\leq \left(\sup\limits_{m\in \N}\left(\left(1+\frac{\Delta}{m}\right)^\frac{m}{\Delta}\right)\right)^\frac{\Delta}{2}\\
&=e^{\frac{\Delta}{2}},
\end{align*}
using a suitable characterization of Euler's number $e$.
As all Graph norms of $A(t)$ are uniformly equivalent by Lemma \ref{equinorm} and $\Delta$ is independent of $s$, $t$ and $n$, we can conclude that
$Q_n(s,t)$ is bounded in $B(W)$ by some constant $C$ independent of $s$, $t$ and $n$.
\end{proof}

We are interested in taking the limit for $n\rightarrow\infty$. For this we calculate (using that the $Q_n(s,t)$ are isometries on $H$ and uniformly bounded on $W$) for $x\in W$
\begin{align*}
\norm{Q_n(t,s)x-Q_m(t,s)x} &=\norm{Q_n(t,s)Q_m(s,s)x-Q_n(t,t)Q_m(t,s)x}\\
&=\norm{\int\limits_s^t \pd{r} (Q_n(t,r)Q_m(r,s)x)dr}\\
&=\norm{\int\limits_s^t Q_n(t,r)(iA_m(r)-iA_n(r))Q_m(r,s)xdr}\\
&\leq\left|C\int\limits_s^t ||A_m(r)-A_n(r)||_{W \rightarrow H}dr\right|||x||_W.
\end{align*}
Here we used Lemma \ref{prodiff} for existence and form of the derivative. Arguments like this one will be used several times in his section to estimate differences. 
As $A:[0,T]\rightarrow B(W,H)$ is a continuous function on a compact domain, it is uniformly continuous. Thus $A_n(r)$ converges to $A(r)$ uniformly in $r$, whence
$$||A_n(r)-A_m(r)||_{W\rightarrow H}$$
converges to 0 for $n,m\rightarrow \infty$, uniformly in $r$. 

By the above estimate, we can conclude that $Q_n(t,s)x$ converges to some limit $Q(t,s)x$ uniformly in $s$ and $t$. As all $Q_n(s,t)$ are isometries with respect to the norm of $H$, the operator $Q(s,t)$ defined in this way is also an isometry, hence it extends to a map from $H$ to $H$. For $y\in H$ and $\epsilon >0$, choose $x\in W$ with $||x-y||<\frac{\epsilon}{3}$. For $n$ large enough such that $||Q(t,s)x-Q_n(t,s)x||<\frac{\epsilon}{3}$ for all $s$ and $t$, we get
\begin{align*}
&||Q(t,s)y-Q_n(t,s)y||\\
&\leq ||Q(t,s)(x-y)||+||Q_n(t,s)(x-y)||+||Q(t,s)x-Q_n(t,s)x||\\
&\leq 2||x-y||+\frac{\epsilon}{3}\\
&\leq\epsilon.
\end{align*}
Thus $Q_n(t,s)y$ converges to $Q(t,s)y$ uniformly in $t$ and $s$. We can conclude that $Q(t,s)$ is strongly continuous in $B(H)$. 

As $Q_n$ satisfies conditions \ref{a},\ref{b} and \ref{c} from Theorem \ref{evop}, so does $Q$.\\\\
To show condition 4. and the second part of condition 6., we need the following lemma.

\begin{lemma}
\label{inteq}
If $E,F:[0,T]\times [0,T]\rightarrow B(H)$ are strongly continuous families of operators, there is a unique, strongly continuous family $V:[0,T]\times [0,T]\rightarrow B(H)$ such that for all $x\in H$, we have
$$V(t,s)x=E(t,s)x+\int\limits_s^t V(t,r)F(r,s)xdr$$
\end{lemma}
\begin{proof}
Let 
$$V_0:=E$$ 
and define inductively
$$V_m(t,s)x:=\int\limits_s^t {V_{m-1}}(t,r)F(r,s)xdr.$$
Claim:
$$||V_m(t,s)||\leq CK^m\frac{(t-s)^m}{m!},$$
for all $s,t\in [0,T]$, with 
$$C:=sup_{s,t\in [0,T]}||E(t,s)||$$
and 
$$K:=sup_{s,t\in [0,T]}||F(t,s)||.$$
Note that $K$ and $C$ are finite by Banach-Steinhaus. We will prove the claim by induction. $||V_0(t,s)||\leq C$ holds by definition. Assume the claim holds for all $m<n$. Then we have
\begin{align*}
||V_n(t,s)x||\leq \int\limits_s^t \norm{V_{n-1}(t,r)}\norm{F(r,s)}||x||dr\\
\leq \int\limits_s^t CK^{n-1}\frac{(t-r)^{n-1}}{(n-1)!}K||x||dr
&=CK^{n}\frac{(t-s)^n}{n!}||x||,
\end{align*}
concluding the induction and proving the claim. The claim implies that the series
$$V(t,s):=\sum\limits_{m=0}^\infty V_m(t,s)x,$$
converges uniformly in $t$, $s$ and $x$ for fixed $||x||$. The operators $V(t,s)$ defined this way are thus a strongly continuous family of bounded operators. Moreover, uniform convergence  allows us to interchange the infinite sum with an integral over a compact interval. Thus we get
\begin{align*}
E(t,s)+\int\limits_s^t V(t,r)F(r,s)xdr&=V_0(t,s)x+\sum\limits_{m=0}^\infty\int\limits_s^t V_m(t,r)F(r,s)xdr\\
&=V_0(t,s)x+\sum\limits_{m=0}^\infty V_{m+1}(t,s)x\\
&=V(t,s)x,
\end{align*}
concluding the proof of existence. For uniqueness, assume $U$ was another solution to the equation.
Then we have
$$U(t,s)x-V(t,s)x=\int\limits_s^t (U(t,r)-V(t,r))F(r,s)xdr.$$
The norm of a strongly continuous family is lower semi-continuous and hence measurable. Since $U(t,r)$ and $V(t,r)$ are uniformly bounded by Banach-Steinhaus, the norm of their difference is bounded and hence integrable.
Thus we can estimate
$$||U(t,s)-V(t,s)||\leq K\int \limits_s^t ||V(t,r)-U(t,r)||dr.$$
By Gronwalls inequality, this implies $U(t,s)-V(t,s)=0$, proving uniqueness.
\end{proof}
With the help of this, we shall now show the following lemma.
\begin{lemma}
$Q(t,s)$ maps $W$ to $W$ boundedly and $(t,s)\mapsto Q(t,s)|_W$ is a strongly continuous family in $B(W)$.
\end{lemma}
\begin{proof}
Let $R(t):=(A(t)+i)^{-1}$. As $R$ is strongly differentiable as a map $H\rightarrow W$ by \ref{idiff}, the composition $Q_n(t,r)R(r)Q_n(r,s)$ is strongly differentiable in $B(W,H)$ as a function of $r$, except at $r=t_k^n$, where it is still strongly continuous. We can calculate for $x\in H$
\begin{align*}
R(t)Q(t,s)x-Q(t,s)R(s)x&=\limn R(t)Q_n(t,s)x-Q_n(t,s)R(s)x\\
&=\limn \int\limits_s^t\pd{r}(Q_n(t,r)R(r)Q_n(r,s)x)dr\\
&=\limn \int\limits_s^t Q_n(t,r)(-iA(r)R(r)+R'(r)+R(r)iA(r))Q_n(r,s)xdr\\
&=\limn \int\limits_s^t Q_n(t,r)R'(r)Q_n(r,s)xdr
\end{align*}
As the $Q_n$ are unitary in $B(H)$ and $R'(r)$ is strongly continuous and hence uniformly bounded, the integrand is bounded, so we can apply dominated convergence to pull the limit into the integral. We get
\begin{align*}
R(t)Q(t,s)x-Q(t,s)R(s)x&=\int\limits_s^t \limn Q_n(t,r)R'(r)Q_n(r,s)xdr\\
&=\int\limits_s^t Q(t,r)R'(r)Q(r,s)xdr\\
&=-\int\limits_s^t Q(t,r)R(r)A'(r)R(r)Q(r,s)xdr
\end{align*}
which we can reformulate as
$$Q(t,s)R(s)x=R(t)Q(t,s)x+\int\limits_s^t Q(t,r)R(r)A'(r)R(r)Q(r,s)xdr.$$
Let $V(t,s)$ be the unique family satisfying 
$$V(t,s)x=Q(t,s)x+\int\limits_s^t V(t,r)A'(r)R(r)Q(r,s)xdr$$
for all $x\in H$, whose existence is guaranteed by Lemma \ref{inteq} with $E:=Q$ and $F(t,s):=A'(t)R(t)Q(t,s)$.
Then both $H(t,s)=Q(t,s)R(s)$ and $H(t,s)=R(t)V(t,s)$ are solutions to
$$H(t,s)=R(t)Q(t,s)+\int\limits_s^t H(t,r)A'(r)R(r)Q(r,s)xdr.$$
By the uniqueness in Lemma \ref{inteq} (this time with $E(t,s)=R(t)Q(t,s)$), this means
$$Q(t,s)R(s)=R(t)V(t,s).$$
Reformulating this as
$$Q(t,s)|_W=R(t)V(t,s)R(s)^{-1}$$
yields the desired result, as $R$, $V$ and $R^{-1}$ are strongly continuous in $B(H,W)$, $B(H)$ and $B(W,H)$ respectively.
\end{proof}
The last remaining condition to check is that $Q$ has the desired strong derivatives.
\begin{lemma}
$Q(t,s)|_W$ is strongly continuously differentiable in both $s$ and $t$ as a function into $B(W,H)$, with derivatives
$$\pdt Q(t,s)=iA(t)Q(t,s)$$ 
and
$$\pds Q(t,s)=-Q(t,s)iA(s).$$
\end{lemma}
\begin{proof}
The goal is to show that at $t=s$, $Q(t,s)$ has the same derivatives as \newline $exp((t-s)iA(t))$. For this we need to estimate the difference around $s=t$. We have for $x\in W$:
\begin{align*}
exp((t-s)iA(t))x-Q_n(t,s)x&=-\int\limits_s^t\pd{r}(exp((t-r)iA(t))Q_n(r,s)x)dr\\
&=-\int\limits_s^t exp((t-r)iA(t))i(-A(t)+A_n(r))Q_n(r,s)xdr\\
\end{align*}
Using that $exp((t-s)iA(t))$ is an isometry on $H$ and $||Q_n(t,s)||_{W \rightarrow W}$ has a uniform bound C by Lemma \ref{Qbound}, we have
$$||exp((t-s)iA(t))x-Q_n(t,s)x||\leq C||x||\int\limits_s^t||A_n(r)-A(t)||_{W \rightarrow H}dr.$$
Taking the limit and using dominated convergence (as $A(t)$ and hence also $A_n(t)$ are uniformly bounded in $B(W,H)$), we obtain
\begin{align*}
||exp((t-s)iA(t))x-Q(t,s)x||&\leq C||x||\int\limits_s^t\limn||A_n(r)-A(t)||_{W \rightarrow H}dr\\
&=C||x||\int\limits_s^t||A(r)-A(t)||_{W \rightarrow H}dr.
\end{align*}
As $A$ is norm continuous in $B(W,H)$, we know that  
$$I_{max}(s,t):=\sup\limits_{r\in[s,t]}||A(r)-A(t)||_{W \rightarrow H}$$
 converges to 0 for $|t-s|\rightarrow 0$, both for fixed $t$ and for fixed $s$. Thus 
$$|t-s|^{-1}||exp((t-s)iA(t))x-Q(t,s)x||\leq C||x||I_{max}(s,t)$$
converges to 0 as well, so we have
$$\pdt (exp((t-s)iA(t))x-Q(t,s)x)|_{t=s}= \lim\limits_{t\rightarrow s}\ (t-s)^{-1}(exp((t-s)iA(t))x-Q(t,s)x)=0$$
and
$$\pds (exp((t-s)iA(t))x-Q(t,s)x)|_{s=t}= \lim\limits_{s\rightarrow t}\ (s-t)^{-1}(exp((t-s)iA(t))x-Q(t,s)x)=0.$$
We can conclude that $Q(t,s)$ is strongly differentiable both in $s$ and in $t$ at $t=s$ with
$$\pdt Q(t,s)|_{t=s}=\pdt exp((t-s)iA(t))|_{t=s}=iA(s)$$
and
$$\pds Q(t,s)|_{s=t}=\pds exp((t-s)iA(t))|_{s=t}=-iA(t).$$
The result for $s\neq t$ now follows, as 
$$\pdt Q(t,s)=\pd{r}Q(r,t)Q(t,s)|_{r=t}=iA(t)Q(t,s)$$
and
$$\pds Q(t,s)=\pd{r}Q(t,s)Q(s,r)|_{r=s}=-Q(t,s)iA(s).$$
\end{proof}
This concludes the proof of Theorem \ref{evop}.
\end{proof}
\begin{thm}
\label{uniq}
For every $s\in [0,T]$ and $Q$ as constructed above, $U(t)=Q(t,s)|_W$ is the only strongly continuous family of operators in $B(W)$ that is strongly differentiable in $B(W,H)$ and satisfies 
$$U(s)=I$$
and
$$\ddt U(t)=iA(t)U(t).$$
In particular, $Q$ is the unique family satisfying the conditions of Theorem \ref{evop}.
\end{thm}
\begin{proof}
Assume $U$ is a family satisfying the conditions given above. Then we can calculate for any $x\in H$:
\begin{align*}
U(t)x-Q(t,s)x&=\int\limits_s^t\pd{r}Q(t,r)U(r)xdr\\
&=\int\limits_s^t Q(t,r)(-iA(r)+iA(r))U(r)xdr\\
&=0.
\end{align*}
\end{proof}
\subsection{The Cauchy Property}
$Q(t,s)x$ can be thought of as taking the function $f:[0,T]\rightarrow H$ satisfying $f(s)=x$ and $Df=0$ and evaluating it at $t$, where in this case $f=Q(\cdot,s)$. This can be generalized in the following definition
\begin{df}
Denote by $ev_t$ the evaluation map at $t$, i.e.
$$ev_t(f):=f(t).$$
We say that a closed unbounded operator $R$ in $\LH$ has the Cauchy property, if $\Dom(R)$ is a subset of $C([0,T],H)$ (with maximum norm) with bounded inclusion and for all $t\in [0,T]$ the map
$$R\oplus ev_t : \Dom(R)\rightarrow \LH \oplus H$$
is an isomorphism. 
\end{df}
\head{Remark:} Another way of thinking about the Cauchy property is to say that for each $s\in[0,T]$, $x\in H$ and $g\in\LH$, the initial value problem
$$Rf=g$$
$$f(s)=x$$
has a unique solution (in $\Dom(R)$).
For any $R$ with the Cauchy property, the evolution operator $Q(t,s):H\rightarrow H$ for $t,s\in[0,T]$ can be defined by
$$Q(t,s)x=ev_t(R\oplus ev_s)^{-1}(0,x).$$ If $R=D$ has the Cauchy property, this agrees with the previous definition of $Q$, as
$$(D\oplus ev_s)Q(\cdot,s)x=(0,x)$$
and hence
$$ev_t(R\oplus ev_s)^{-1}(0,x)=ev_t(Q(\cdot,s)x)=Q(t,s)x.$$

We will now show that $D$ has the Cauchy property, starting with the following lemma:
\begin{lemma}
\label{cont}
$\Dom(D)$ (with graph norm) is a subspace of $C^0([0,T],H)$ and the inclusion is continuous.
\end{lemma}
\head{Remark:} This implies that evaluation at some point is well-defined and bounded on $\Dom(D)$.
\begin{proof}

For $f\in C^1([0,T],W)$ and $t\in[0,T]$, we want to estimate $f(t)$ in terms of $||f||_{L^2}$ and $||Df||_{L^2}$. The idea is, that a bound on the $L^2$ norm prevents the function from being large everywhere and $Df$ gives us a bound on the derivative of $||f||$, so $f$ cannot be large on a small region only. The first observation is that $D$ is as good as the derivative when estimating the change in the norm. As $A(t)$ is self-adjoint, $\<x,A(t)x\>$ is real for all $t\in[0,T]$ and $x\in W$. Thus we have
\begin{align*}
\pdt ||f(t)||^2&=2Re(\<f(t),f'(t)\>)\\
&=2Re(\<f(t),f'(t)\>)-2Re(i\<f(t),A(t)f(t)\>)\\
&=2Re(\<f(t),f'(t)-iA(t)f(t)\>)\\
&=2Re(\<f(t),Df(t)\>).
\end{align*}
We can now estimate for any $t\in [0,T]$:
\begin{align*}
T\norm{f(t)}^2-\norm{f}_{L^2}^2
&= \int\limits_0^T\norm{f(t)}^2-\norm{f(s)}^2ds\\
&= \int\limits_0^T\int\limits_s^t\pd{r}\norm{f(r)}^2drds\\
&= \int\limits_0^T\int\limits_s^t2Re(\<f(r),Df(r)\>)drds\\
&\leq \int\limits_0^T2|Re\<\chi_{[s,t]}f,Df\>_{L^2}|ds\\
&\leq\int\limits_0^T2||\chi_{[s,t]}f||_{L^2}\cdot||Df||_{L^2}ds\\
&\leq\int\limits_0^T2||f||_{L^2}||Df||_{L^2}ds\\
&\leq T(\norm{f}_{L^2}^2+\norm{Df}_{L^2}^2)\\
&\leq T\norm{f}_{D}^2,
\end{align*}
We get
$$\norm{f(t)}^2\leq T^{-1} \norm{f}_{L^2}^2+\norm{f}_{D}^2\leq (1+T^{-1})\norm{f}_{D}^2.$$
Since this is true for all $T$, we have
$$||f||_{C^0}:=\max\limits_{t\in[0,T]}||f(t)||\leq (1+T^{-1})^\frac{1}{2}\norm{f}_{D}.$$
As $C^1([0,T],W)$ is dense in $\Dom(D)$, its inclusion into $C^0([0,T],H)$ extends uniquely to a continuous map on $\Dom(D)$. As both convergence in $\Dom(D)$ and in $C^0([0,T],H)$ imply convergence in $\LH$, the limits in both spaces agree, so $\Dom(D)\subseteq C^0([0,T],H)$ and the continuous extension is the inclusion of $\Dom(D)$ into $C^0([0,T],H)$.
\end{proof}

\begin{thm}
$D$ has the Cauchy property. More specifically
$$F_s(g,x)(t):=Q(t,s)x+\int\limits_s^tQ(t,r)g(r)dr$$
defines an inverse for $D\oplus ev_s$.
\end{thm}
\begin{proof}
For $g\in C^1([0,T],W)$ and $x\in H$, we get (using Lemma \ref{diffint} for differentiating the integral) 
\begin{align*}
\ddt F_s(g,x)(t)&=\pdt Q(t,s)x+\pdt\int\limits_s^tQ(t,r)g(r)dr\\
&=iA(t)Q(t,s)x+\int\limits_s^t \pdt Q(t,r)g(r)dr+Q(t,t)g(t)\\
&=iA(t)Q(t,s)x+iA(t)\int\limits_s^tQ(t,r)g(r)dr+g(t)\\
&=iA(t) F_s(g,x)(t)+g(t),
\end{align*}
where $A$ can be pulled out of the integral, as the integral also exists in $W$.
Thus $DF_s(g,x)=g$. As
$$F_s(g,x)(s)=Q(s,s)x=x,$$
we have
$$(D\oplus ev_s)\circ F_s|_{C^1([0,T],W)}=id.$$
Conversely, for $f\in C^1([0,T],W)$, we have
\begin{align*}
F_s(Df,f(s))(t)&=Q(t,s)f(s)+\int\limits_s^t Q(t,r)\left(\frac{d}{dr}f(r)-iAf(r)\right)dr\\
&=Q(t,s)f(s)+\int\limits_s^t \left(Q(t,r)\frac{d}{dr}f(r)+\left(\pd{r}Q(t,r)\right)f(r)\right)dr\\
&=Q(t,s)f(s)+\int\limits_s^t \pd{r}\left(Q(t,r)f(r)\right)dr\\
&=Q(t,s)f(s)+Q(t,t)f(t)-Q(t,s)f(s)\\
&=f(t),
\end{align*}
whence
$$F_s\circ(D\oplus ev_s)|_{C^1([0,T],W)}=id.$$
Now it remains to extend this to the whole spaces. If $x_n$ is a sequence in $W$ converging to a limit $x$ in $H$ and $\phi$ is in $C^\infty([0,T],\C)$, $\phi x_n$ converges to $\phi x$ in $\LH$. Thus $S([0,T],W)$ (as defined in Lemma \ref{Sdense}) is dense in $S([0,T],H)$, which is dense in $\LH$ by Lemma \ref{Sdense}. We can conclude that $S([0,T],W)$ is dense in $\LH$. As $C^1([0,T],W)$ contains $S([0,T],W)$, it is also dense in $\LH$. Moreover, $C^1([0,T],W)$ is dense in $\Dom(D)$ by definition.
Using that $Q$ is an isometry, we obtain for $(f,x)\in C^1([0,T],W)\oplus H$:
\begin{align*}
||F_s(f,x)||_{D}^2&=||F(f,x)||_{L^2}^2+||D F(f,x)||_{L^2}^2\\
&\leq\norm{Q(\cdot,s)x}_{L^2}^2+\norm{\int\limits_s^\cdot Q(\cdot,r)f(r)dr}_{L^2}^2+\norm{f}_{L^2}^2\\
&\leq\norm{x}_{L^2}^2+\norm{\int\limits_t^\cdot\norm{f(r)}dr}_{L^2}^2+\norm{f}_{L^2}^2\\
&\leq T||x||^2+\norm{\int\limits_0^T\norm{f(r)}dr}_{L^2}^2+\norm{f}_{L^2}^2\\
&\leq T||x||^2+T||f||_{L_1}^2+||f||_{L_2}^2\\
&\leq T||x||^2+CT||f||_{L_2}^2+||f||_{L_2}^2\\
&\leq C'||(f,x)||_{\LH\oplus H}^2,
\end{align*}
for some constants $C$ and $C'$ (as $\LH$ is boundedly included in $L^1([0,T],H)$).
Thus $F_s|_{C^1([0,T],W)\oplus H}$ extends uniquely to a bounded map $\LH\oplus H\rightarrow \Dom(D)$. As both this extension and $F_s$ are bounded as maps into $\LH$, they coincide, so $F_s$ maps boundedly into $\Dom(D)$. $D\oplus ev_s$ is bounded by definition of the operator norm and by Lemma \ref{cont}. 

Thus $F_s(D\oplus ev_t)$ and $(D\oplus ev_t)F_s$ are bounded maps that are the identity on a dense subset, so they must be the respective identities.
\end{proof}

\section{The Index and the Spectral Flow}
\label{s5}
In the following, we will use the splitting of $H$ in positive and negative spectral subspaces of $A(t)$:
\begin{df}
For $t\in[0,T]$, let 
$$P_{\geq 0}(t):=P_{[0,\infty)}(t)=I-P_{<0}(t).$$
Let $$H_{<0}(t):=\Ran(P_{<0}(t))$$ and
$$H_{\geq 0}(t):=\Ran(P_{\geq 0}(t))$$
\end{df}
We define APS boundary conditions:
\begin{df}
For $s<t\in[0,T]$, let $(D|_{[s,t]})_{\APS}$ be the restriction of $D|_{[s,t]}$ to $$\Dom((D|_{[s,t]})_{\APS}):=\{f\in \Dom(D)|f(s)\in H_{<0}(s),f(t)\in H_{\geq 0}(t)\}.$$
Let 
$$D_{\APS}:=(D|_{[0,T]})_{\APS}$$
\end{df}
The following evolved spectral projection will be crucial for proving the main theorem.
\begin{df}
Let 
$$\hat P_{<a}(t):=Q(0,t)P_{<a}(t)Q(t,0).$$
Let $\hat P_{<a}(t)_r$ be the restriction 
$$\hat P_{<a}(t)_r:=[\hat P_{<a}(t):{H_{<0}(0)}\rightarrow{Q(0,t)H_{(-\infty,a)}(t)}].$$
\end{df}
Note that $\hat P_{<a}(t)$ is the projection onto $Q(0,t)H_{(-\infty,a)}(t)$. Thus $\hat P_{<a}(t)_r$ is defined so that, by definition of the index of pairs,  its index is that of the pair $(P_{<0}(0),\hat P_{<a}(t))$. Also note that $\hat P_{<0}(0)=P_{<0}(0)$.

In this section, the main theorem will be shown, stating that under a certain condition we have 
$$\ind(D_{\APS})=\sfl(A).$$
For this purpose, we will first show that the index of $D_{\APS}$ and that of the pair $(\hat P_{<0}(0),\hat P_{<0}(T))$ coincide (if either one is defined). Then we will show that those projections are spectral projections of a strongly differentiable family $\hat A$ of operators that has the same spectral flow as $A$. We can then use Theorem \ref{flowind} for $\hat A$ to conclude the proof of the equality above.

\subsection{The Index of $D_{\APS}$}
The first step towards the index formula is to investigate the Fredholm index of $D_{\APS}$. This corresponds to section 3 in [B\"aSt].
\begin{thm}
\label{inD}
Write $Q$ for $Q(T,0)$.
The kernel of $D_{\APS}$ is:
$$\Ker(D_{\APS})\cong H_{<0}(0)\cap Q^{-1}H_{\geq 0}(T)=\Ker(\hat P_{<0}(T)_r)$$
$D_{\APS}$ has closed range if and only if $P_{<0}(t)_r$ has closed range and in this case
$$\Coker(D_{\APS})\cong H_{\geq 0}(0)\cap Q^{-1}H_{<0}(T)\cong \Coker(\hat P_{<0}(T)_r).$$
In particular, $D_{\APS}$ is Fredholm with index $k$ if and only if  $(P_{<0}(0),\hat P_{<0}(T))$ is a Fredholm pair with index $k$.
\end{thm}
\head{Remark:} By replacing $A$ by $A|_{[0,t]}$, we obtain for any $t\in[0,T]$ that $(D|_{[0,t]})_{\APS}$ is Fredholm with index $k$ if and only if $\hat P_{<0}(t)_r$ is.
\begin{proof}
We have
\begin{align*}
\Ker(D_{\APS})&=\{f\in \Dom(D)| Df=0,f(0)\in H_{<0}(0), f(T)\in H_{\geq 0}(T)\}\\
&\cong\{f(0)\in H_{<0}(0)|Qf(0)\in H_{\geq 0}(T)\}\\
&=H_{<0}(0)\cap Q^{-1}H_{\geq 0}(T)\\
&=\Ker(\hat P_{<0}(t)_r),
\end{align*}
where in the second line, we use that $f\mapsto f(0)$ is an isomorphism due to the Cauchy property and that $Df=0$ implies $f(t)=Q(t,0)f(0)$. In the last line, we used that $Q^{-1}H_{\geq 0}(T)$ is the orthogonal complement of $Q^{-1}H_{<0}(T)$, onto which $\hat P_{<0}(t)$ projects.

For $g\in \LH$ define
$$E(g):=ev_T(D\oplus ev_0)^{-1}(g,0).$$
This is in some sense the counterpart to $Q$: Rather than starting with $f(0)=x$ for some $x\in H$ and evolving subject to $Df=0$, one starts with $f(0)=0$ and evolves subject to $Df=g$. By linearity, we have
$$ev_T(D\oplus ev_0)^{-1}(g,x)=E(g)+Qx$$

To see when $D_{\APS}$ has closed range and to determine its cokernel, we need to characterize its range. 

For $g\in \LH$, we have the following chain of equivalences:
\begin{align*}
g\in \Ran(D_{\APS})&\Leftrightarrow \exists f\in \Dom(D): f(0)\in H_{<0}(0)\wedge f(T) \in H_{\geq 0}(T) \wedge Df=g\\
&\Leftrightarrow \exists f(0)\in H_{<0}(0): ev_T(D\oplus ev_0)^{-1}(g,f(0))\in H_{\geq 0}(T)\\
&\Leftrightarrow\exists f(0)\in H_{<0}(0):\exists z\in H_{\geq 0}(T): E(g)+Q(f(0))=z\\
&\Leftrightarrow\exists x\in QH_{<0}(0):\exists z\in H_{\geq 0}(T): E(g)=z-x\\
&\Leftrightarrow E(g)\in QH_{<0}(0)+ H_{\geq 0}(T)\\
\end{align*}
Defining 
$$V:=QH_{<0}(0)+ H_{\geq 0}(T),$$
we get
$$\Ran(D_{\APS})=\{g\in\LH|E(g)\in V\}=E^{-1}(V).$$

\head{Claim}: $E:\LH \rightarrow H$ is surjective.\\
We need to show that functions in $\Dom(D)$ that vanish at 0 can take any value at $T$.  For $z\in H$ choose $f\in \Dom(D)$ with $f(T)=z$ (a possible choice is  $f(t)=Q(t,T)z$) and let
$\phi(t):=\frac{t}{T}$. Multiplication with $\phi$ preserves $\Dom(D)$ as it is bounded on $\Dom(D_0)$ with respect to $||\cdot||_D$. Thus, $\phi f$ is in $\Dom(D)$, with $\phi(0)f(0)=0$ and $\phi(T)f(T)=z$. We get 
$$E(D(\phi f))=ev_T\circ(D\oplus ev_0)^{-1}(D(\phi f),0)=ev_T(\phi f)=z.$$
As $z$ was arbitrary, $E$ is surjective.

As $E|_{\Ker(E)^\bot}$ is an isomorphism, we can define 
$$\iota:=(E|_{\Ker(E)^\bot})^{-1}.$$
By the bounded inverse theorem, $\iota$ is continuous.
We can conclude that
$$\Ran(D_{\APS})=E^{-1}(V)=\iota(V)\oplus \Ker(E)$$
is closed if and only if $V$ is closed.

Next we relate closedness of $V$ to that of $\Ran(\hat P_{<0}(T)_r)$.

For $x\in QH_{<0}(0)$ and $y\in H_{\geq 0}(T)$, we have
$$x+y=P_{<0}(T)x+(P_{\geq 0}(T)x+y)\in P_{<0}(T)QH_{<0}(0)+H_{\geq 0}(T)$$
and conversely
$$P_{<0}(T)x+y=x+(y-P_{\geq 0}(T)x)\in V.$$
Thus 
$$V=P_{<0}(T)QH_{<0}(0)+H_{\geq 0}(T).$$
As $H_{\geq 0}(T)$ is closed and the summands are orthogonal, $V$ is closed if and only if $P_{<0}(T)QH_{<0}(0)$ is closed. This is closed if and only if
$$Q^{-1}P_{<0}(T)QH_{<0}(0)=\Ran(\hat P_{<0}(T)_r)$$
is closed. Putting things together, we get that $D_{\APS}$ has closed range if and only if $\hat P_{<0}(T)_r$ has closed range.

Now assume that $\Ran(D_{\APS})$ and hence $V$ is closed.
Let $P$ be the orthogonal projection onto $V^\bot=QH_{\geq 0}(0)\cap H_{<0}(T)$.
$$P\circ E: \LH\rightarrow QH_{\geq 0}(0)\cap H_{<0}(T)$$
is a surjective map with kernel
$$\Ker(P\circ E)=E^{-1}(\Ker(P))=E^{-1}(V)=\Ran(D_{\APS}).$$ 
This means
$$QH_{\geq 0}(0)\cap H_{<0}(T)=\Ran(P\circ E)\cong \Dom(P\circ E)/\Ker(P\circ E)=\LH/\Ran(D_{\APS}).$$
We can conclude:
$$\Coker(D_{\APS})\cong QH_{\geq 0}(0)\cap H_{<0}(T)\cong H_{\geq 0}(0)\cap Q^{-1}H_{<0}(T).$$
Moreover, for every $x\in Q^{-1}H_{<0}(T)$, $x$ is in $H_{\geq 0}(0)=H_{<0}(0)^\bot$ if and only if for all $y\in H_{<0}(0)$
$$0=\<x,y\>=\<\hat P_{<0}(T)x,y\>=\<x,\hat P_{<0}(T)y\>.$$
Thus we can conclude
\begin{align*}
&H_{\geq 0}(0)\cap Q^{-1}H_{<0}(T)\\
&=\left(\hat P_{<0}(T)H_{<0}(0)\right)^\bot\cap Q^{-1}H_{<0}(T)\\
&=\Ran(\hat P_{<0}(T)_r)^\bot\ \ \ \ \ (in\ Q^{-1}H_{<0}(0))\\
&\cong \Coker(\hat P_{<0}(T)_r).
\end{align*}
\end{proof}
In [B\"aSt], the operator $Q_{--}$ defined below is used instead of the pair $(P_{<0}(0),\hat P_{<0}(t))$.  When considering indices, the two are equivalent: 
\begin{thm}
For every $t\in[0,T]$, the operator 
$$Q_{--}(t,0):=[P_{<0}(t)Q(t,0):{H_{<0}(0)}\rightarrow{H_{<0}(t)}]$$
is Fredholm with index k if and only if $(P_{<0}(0),\hat P_{<0}(t))$ is a Fredholm pair with index k.
\end{thm}
\begin{proof}
$\hat P_{<0}(t)$ is the projection onto $Q(0,t)H_{<0}(t)$.
We have
\begin{align*}
\hat P_{<0}(t)_r&=[(Q(0,t)P_{<0}(t)Q(t,0)):{H_{<0}(0)}\rightarrow{Q(0,t)H_{<0}(t)}]\\
&=[Q(0,t):{H_{<0}(t)}\rightarrow{Q(0,t)H_{<0}(t)}]\circ Q_{--}(t,0)
\end{align*}
As $[Q(0,t):{H_{<0}(t)}\rightarrow{Q(0,t)H_{<0}(t)}]$ is an isomorphism,
$\hat P_{<0}(t)_r$
is Fredholm with index k if and only if $Q_{--}(t,0)$ is.
\end{proof}

In some special cases , this can be immediately used to show $\ind(D_{\APS})=\sfl(A)$:

\begin{thm}
If $A$ is uniformly bounded from below and $\spec(A)$ is  discrete with only finite multiplicities, $D_{\APS}$ is Fredholm and
$$\ind(D_{\APS})=\sfl(A)$$
\end{thm}
\head{Remark:} This covers the special case that $H$ is finite dimensional.
\begin{proof}
Let $a$ be a lower bound for all $A(t)$. Then 
$H_{<0}(t)=H_{[-a,0)}(t)$
is finite dimensional all $t$. In particular,  $\hat P_{<a}(T)_r$ is an operator between finite dimensional spaces, so it is Fredholm, with its index being the difference of the dimensions of domain and range. Thus we have 
\begin{align*}
\ind(\hat P_{<a}(T)_r)&=\Dim(H_{[-a,0)}(0))-\Dim(Q(0,T)H_{[-a,0)}(T))\\
&=\Dim(H_{[-a,0)}(0))-\Dim(H_{[-a,0)}(T))\\
&=\sfl(A),
\end{align*}
by Lemma \ref{minflow}.
By Theorem \ref{inD}, this implies the claim.
\end{proof}

\head{Remark:} Note that the specific form of $D$ was not used anywhere in this section yet. The proof works for any operator $D$ that has the Cauchy property and such that $ev_0\oplus ev_T:\Dom (D)\rightarrow H\oplus H$ is surjective, with $Q$ being the associated evolution operator. The dependence of the index on $\sfl(A)$ can thus be seen to come mainly (in the special case above, only) from the APS boundary conditions with respect to $A$, while the appearance of $A$ in the definition of $D$ is only relevant for the evolution operator, which has to be "nice enough" to not destroy the equality. In particular, we could look at $D=\overline {\ddt-iA_1}$ and take APS boundary conditions with respect to the spectral decomposition of $A_2$, where $A_1$ and $A_2$ satisfy the conditions for $A$. Everything we have done so far would work, with the evolution operator being defined with respect to $A_1$ and all spectral projections (and the spectral flow) with respect to $A_2$. In fact, for the conclusion of the proof of the main theorem, we would only need that $A_1(t)$ and $A_2(t)$ commute. As a special case of this, we could consider $A_1=A$ and $A_2=-A$. This would give the same operator $D$, but with "anti-APS" boundary conditions, i.e. $f(0)\in H_{(0,\infty)}(0)$ and $f(T)\in H_{(-\infty,0]}(T)$. We would then get the equality $$\ind(D_{\operatorname{anti-APS}})=\ind(P_{>0}(0), \hat{P}_{>0}(t))=sf(-A).$$

\subsection {The Main Theorem}
Let 
$$\hat A(t):=Q(0,t)A(t)Q(t,0).$$
As functional calculus is equivariant under conjugation with isometries,
we have
$$\chi_{(-\infty,a)}(\hat A(t))=Q(0,t)\chi_{(-\infty,a)}(A(t))Q(t,0)=\hat P_{<a}(t).$$
For every $t\in[0,T]$, $\hat A(t)$ is self-adjoint and Fredholm, with domain $W$ (as $Q(t,0)^{-1}(W)=W$). In order to show that $\hat A$ has all the properties required of $A$, the only thing we need to prove is that $\hat A$ is strongly differentiable in $B(W,H)$.
\begin{lemma}
\label{hatdiff}
$\hat A: J\rightarrow B(W,H)$ is strongly continuously differentiable with derivative
$$\hat A'(t) = Q(0,t)A'(t)Q(t,0)$$
\end{lemma}
\head{Remark:} An informal calculation ignoring domains would read
\begin{align*}
&\ddt(Q(0,t)A(t)Q(t,0))\\
&=-Q(0,t)iA(t)A(t)Q(t,0)+Q(0,t)A'(t)Q(t,0)+Q(0,t)A(t)iA(t)Q(t,0)\\
&=Q(0,t)A'(t)Q(t,0).
\end{align*}
However, $Q$ is not differentiable as a map $H\rightarrow H$ or $W\rightarrow W$ and the terms that cancel out contain second powers of $A$, which are not defined on $W$. One would like to restrict to a dense subspace where everything exists, but it is not clear if such a subspace exists. To remedy this, we will consider the inverse. This maps $H$ to $W$, allowing us to restrict the domain and extend the range in such a way that the calculations work rigorously.
\begin{proof}
Let $R(t):=(A(t)-i)^{-1}$ for $t\in [0,T]$. $\hat A(t)$ is differentiable at $t$ if and only if 
$$\hat A(t)-i=Q(0,t)(A(t)-i)Q(t,0)$$ 
is. As $Q(t,0)$ and $Q(0,t)$ are strongly continuously differentiable in $B(W,H)$ and $R(t)$ is strongly continuously differentiable in $B(H,W)$, we get from Lemmas \ref{prodiff} and \ref{idiff} that 
$$(\hat A(t)-i)^{-1}=Q(0,t)R(t)Q(t,0)$$
is strongly differentiable in $B(W,H)$. Its derivative is
\begin{align*}
&\ddt (\hat A(t)-i)^{-1} \\
&=\ddt Q(0,t)R(t)Q(t,0)\\
&=Q(0,t)R(t)iA(t)Q(t,0)-Q(0,t)R(t)A'(t)R(t)Q(t,0)-Q(0,t)iA(t)R(t)Q(t,0)\\
&=-Q(0,t)R(t)A'(t)R(t)Q(t,0).
\end{align*}
As this is strongly continuous in $B(H,W)$, Lemma \ref{strictdiff} implies that $(\hat A(t)-i)^{-1}$ is strongly continuously differentiable in $B(H,W)$. By Lemma \ref{idiff}, $\hat A(t)-i$ and hence $\hat A(t)$ are strongly continuously differentiable, with derivative
\begin{align*}
\hat A'(t)&=\ddt ((\hat A(t)-i)^{-1})^{-1}\\
&=-(\hat A(t)-i)\left(\ddt(\hat A(t)-i)^{-1}\right)(\hat A(t)-i)\\
&=Q(0,t)(A(t)-i)Q(t,0)\circ Q(0,t)R(t)A'(t)R(t)Q(t,0) \circ Q(0,t)(A(t)-i)Q(t,0)\\
&=Q(0,t)A'(t)Q(t,0)
\end{align*}
\end{proof}

With this, we have all the pieces in place now to prove the main theorem.
\begin{thm}
\head{(Main Theorem)}
\label{main}
If $(D|_{[0,t]})_{\APS}$ is Fredholm for all $t\in[0,T]$, we have
$$\ind(D_{\APS})=\sfl(A).$$
\end{thm}
In light of Theorem \ref{inD}, this is equivalent to the following:
\begin{thm}
\head{(Main Theorem II)}
\label{mainII}
If $(P_{<0}(0),\hat P_{<0}(t))$ is Fredholm for all $t\in[0,T]$, we have
$$\ind(D_{\APS})=\sfl(A).$$
\end{thm}
\begin{proof}
$\hat A$ has all properties required of $A$. We can thus apply Theorem \ref{flowind} to it. The spectral projections of $\hat A$ are given by
$$\chi_{(-\infty, 0)}(\hat A(t))=\hat P_{<0}(t).$$
Using theorems \ref{inD} and \ref{conflow}, we can now conclude
\begin{align*}
&\ind (D_{\APS})\\
&\stackrel{\ref{inD}}{=}\ind(P_{<0}(0),\hat P_{<0}(T))\\
&=\ind(\hat P_{<0}(0),\hat P_{<0}(T))\\
&\stackrel{\ref{flowind}}{=}\sfl(\hat A)\\
&\stackrel{\ref{conflow}}{=}\sfl(A).
\end{align*}
\end{proof}
\head{Remark:} By Theorem \ref{sympair}, $(P_{<0}(0),\hat P_{<0}(t))$ is Fredholm if and only if $(P_{\geq 0}(0),Q(0,t)P_{\geq 0}(t)Q(t,0))$ is Fredholm. Thus one can just as well consider positive instead of negative spectral projections.\\\\
As an immediate consequence of the main theorem, we obtain a gluing formula for the index of $D_{\APS}$:
\begin{thm}
If $(D|_{[s,t]})_{\APS}$ is Fredholm for all $s,t\in [0,T]$, we have for all $r<s<t\in[0,T]$:
$$\ind((D|_{[r,t]})_{\APS})=\ind((D|_{[r,s]})_{\APS})+\ind((D|_{[s,t]})_{\APS})$$
\end{thm}
\begin{proof}
We have
\begin{align*}
&\ind((D|_{[r,t]})_{\APS})\\
&=\sfl(A|_{[r,t]})\\
&=\sfl(A|_{[r,s]})+\sfl(A|_{[s,t]})\\
&=\ind((D|_{[r,s]})_{\APS})+\ind((D|_{[s,t]})_{\APS}).
\end{align*}
\end{proof}
\subsection{Sufficient Conditions for the Main Theorem}
It may not be easy to determine a priori whether $D_{\APS}$ is Fredholm. The evolution operator and thus $\hat P_{<0}(t)$ might be difficult to compute as well. The aim of this section is to give other sufficient criteria for the use of the main theorem that will be much less general, but sometimes easier to check.

\begin{thm}
\label{compair}
For $t\in [0,T]$, the following are equivalent:
\begin{enumerate}
\item $Q_{+-}(t,0):=P_{\geq 0}(t)Q(t,0)P_{<0}(0)$ and $Q_{-+}(t,0):=P_{<0}(t)Q(t,0)P_{\geq 0}(0)$ are compact.
\item $(P_{<0}(0), \hat P_{<0}(t))$ is a compact pair, i.e. the difference of the two is compact.
\end{enumerate}
Each of the above conditions imply that $D_{\APS}$ is Fredholm and
$$\ind(D_{\APS})=\sfl(A).$$
\end{thm}
\head{Remark:} The former condition is shown in [B\"aSt] in the setting considered there (Lemma 2.6), so the main theorem is indeed applicable to that special case.
\begin{proof}
The second condition implies that $(P_{<0}(0), \hat P_{<0}(t))$ is Fredholm for all $t\in [0,T]$, by \ref{compfred}. Thus we can apply the second main theorem (\ref{mainII}) to obtain the second part of the claim.

What is left to show is that the two conditions are equivalent.

For $t\in [0,T]$, let $$S_0:=P_{<0}(0)-\hat P_{<0}(t).$$  
and
$$S:=Q(t,0)S_0=Q(t,0)P_{<0}(0)-P_{<0}(t)Q(t,0).$$
Condition 2. holds if and only if $S_0$ is compact, which is true if and only if $S$ is compact. 
Let $S_{\geq 0}:=P_{\geq 0}(t)S$ and $S_{<0}:=P_{<0}(t)S$ .
As compact operators are an ideal, $S_{\geq 0}$ and $S_{<0}$ are compact if $S$ is compact. Conversely,
$$S=S_{\geq 0}+S_{<0}$$
is compact if $S_{\geq 0}$ and $S_{<0}$ are.

Since
$$S_{\geq 0}=P_{\geq 0}(t)Q(t,0)P_{<0}(0)=Q_{+-}(t,0)$$
and
$$S_{<0}=P_{<0}(t)Q(t,0)P_{<0}(0)-P_{<0}(t)Q(t,0)=-P_{<0}(t)Q(t,0)(I-P_{<0}(0))=-Q_{-+}(t,0),$$
the theorem follows.
\end{proof}

\begin{thm}
If $A'(t)$ is compact in $B(W,H)$ for all $t\in[0,T]$, then $D_{\APS}$ is Fredholm and
$$\ind(D_{\APS})=\sfl(A).$$
\end{thm}
\begin{proof}
By Lemma $\ref{hatdiff}$ and Theorem $\ref{dcomp}$, $\hat A(t)- A(0)$ is compact in $B(W,H)$ for every $t\in[0,T]$.
Corollary 3.5 in [LE] states that in this case $\chi_{[0,\infty)}(\hat A(t))-\chi_{[0,\infty)}(A(0))$ is compact in $B(H)$.
Thus 
$$\hat P_{<0}(t)-P_{<0}(0)=(I-\chi_{[0,\infty)}(\hat A(t)))-(I-\chi_{[0,\infty)}(A(0)))=-(\chi_{[0,\infty)}(\hat A(t))-\chi_{[0,\infty)}(A(0)))$$
is compact, so $(P_{<0}(0),\hat P_{<0}(t))$ is a compact pair and hence a Fredholm pair by Lemma \ref{compfred} for all $t\in [0,T]$.
By Theorem \ref{compair}, we get the desired result.
\end{proof}
\head{Remark:} The counterexample in the next chapter shows that it is not sufficient to ask for relative compactness of $A(t)-A(0)$.\\\\
For the next theorem, we will need the following lemma:
\begin{lemma}
\label{test}
If $J$ is an interval and $P:J\rightarrow B(H)$ is a continuous family of commuting projections, $P$ is constant.
\end{lemma}
\begin{proof}
For $s\in J$, there is a neighborhood $U$ of $s$ such that for $t\in U$, we have
$$||P(t)-P(s)||<1.$$
For $t\in U$, we have
$$P(t)P(s)P(t)P(s)=P(t)P(t)P(s)P(s)=P(t)P(s)$$
and
$$(P(t)P(s))^*=P(s)P(t)=P(t)P(s).$$
Thus $P(t)P(s)$ is a (orthogonal) projection as  well. The range of $P(t)P(s)$ is a subspace of the range of $P(t)$. It is closed, as (orthogonal) projections have closed range.

Assume for contradiction that it is a proper subspace, i.e. $P(t)\neq P(t)P(s)$. Then there must be a non-zero $x\in \Ran(P(t))$ such that $x\bot \Ran(P(t)P(s))$, i.e. $x\in \Ker(P(t)P(s))$.
Then
$$P(s)x=P(s)P(t)x=P(t)P(s)x=0.$$
But this implies
$$||P(t)x-P(s)x||=||x||,$$
contradicting $||P(t)-P(s)||<1$.
We can thus conclude the assumption was wrong and
$$P(t)=P(t)P(s).$$
The same reasoning with $s$ and $t$ interchanged yields
$$P(s)=P(s)P(t)=P(t)P(s)=P(t).$$
As this holds for all $t\in U$, for a neighborhood $U$ of every $s\in J$, $P$ is locally constant. As the interval is connected, $P$ is constant.
\end{proof}

\begin{thm}
Let $R(t):=(A(t)-i)^{-1}$. If  $R(s)$ and $R(t)$ commute for all $s,t\in[0,T]$, $D_{\APS}$ is Fredholm with index $\sfl(A)$.
\end{thm}
\head{Remark:} Basically, the assumption is that all $A(t)$ commute with each other. Considering the resolvents is necessary, as $A(t)A(s)$ is not defined on all of $W$.
\begin{proof}
Assume that $\rho:H\rightarrow H$ is a bounded operator such that
$$\rho R(s)=R(s)\rho$$ 
for $s,\in [0,T]$. Note that this implies 
$$\rho(W)=\Ran(\rho R(s))\subseteq\Ran(R(s))=W.$$
 We have
\begin{align*}
\rho A(s)&=\rho (A(s)+i)-i\rho\\
&=(A(s)+i)R(s)\rho (A(s)+i)-i\rho \\
&=(A(s)+i)\rho R(s)(A(s)+i)-i\rho \\
&=(A(s)+i)\rho -i\rho \\
&=A(s)\rho.
\end{align*}
For $\rho=R(t)$, we obtain
$$R(t)A(s)=A(s)R(t).$$
Thus we have for $t\in [0,T]$ and $x\in W$
\begin{align*}
&R(t)Q(t,0)-Q(t,0)R(t)\\
&=\int\limits_0^t\pds (Q(t,s)R(t)Q(s,0))xds\\
&=\int\limits_0^tQ(t,s)\left(-iA(s)R(t)+R(t)iA(s)\right)Q(s,0)xds\\
&=0.
\end{align*}
Setting $\rho =Q(t,0)$ in the above, we obtain
$$Q(t,0)A(t)=A(t)Q(t,0).$$ 
We can conclude
$$\hat A(t)=Q(0,t)Q(t,0)A(t)=A(t)$$
and thus we have 
$$\hat P_{<a}(t)=P_{<a}(t).$$

Let $f(\lambda):=\frac{1}{\lambda-i}$.
As for $t\in [0,T]$
$$P_{<a}(t)=\chi_{(-\infty,a)}(f^{-1}(f(A(t))))=(\chi_{(-\infty,a)}\circ f^{-1})(R(t))$$
the projection $P_{<a}(t)$ commutes with every operator that commutes with $R(t)$. In particular, it commutes with $R(s)$ for $s\in[0,T]$. As $P_{<a}(s)$ commutes with all operators commuting with $R(s)$, we can conclude that $P_{<a}(t)$ and $P_{<a}(s)$ commute.

For $t\in[0,T]$ choose $a\in\Rp$ and a neighborhood $U$ of $t$ such that for all $s\in U$ $H_{[0,a)}(s)$ is finite dimensional and $a\notin \spec(A(s))$  (possible by Lemma \ref{Udisc}). The latter allows us to use  Theorem \ref{kat} to conclude that $P_{<a}$ is continuous on $U$, so by Lemma \ref{test} it is constant.

$$[P_{<0}(s):H_{<0}(0)\rightarrow H_{<0}(s)]$$ is the composition of $$[P_{<a}(s):H_{<0}(0)\rightarrow H_{(-\infty,a)}(s)]$$ with $$[P_{<0}(s):H_{(-\infty,a)}(s)\rightarrow H_{<0}(s)]$$ and the latter is Fredholm, as $H_{[0,a)}(s)$ is finite dimensional. Thus the first of the three is Fredholm if and only if the second is. We can conclude that $(P_{<0}(0),P_{<0}(s))$ is a Fredholm pair if and only if $(P_{<0}(0), P_{<a}(s))$ is. As the latter is the same for all $s\in U$, we can conclude that if $(P_{<0}(0),P_{<0}(s))$ is Fredholm for some $s\in U$ it is Fredholm for all $s\in U$. As this implication holds for neighborhoods of every point in $[0,T]$, the set of $s\in [0,T]$ such that $(P_{<0}(0),P_{<0}(s))$ is Fredholm is open and closed. As the set contains $0$ and $[0,T]$ is connected, we can conclude that $(P_{<0}(0),P_{<0}(s))=(P_{<0}(0),\hat P_{<0}(s))$ is Fredholm for all $s\in [0,T]$. By the second main theorem (\ref{mainII}), this implies the claim.
\end{proof}

\section{A Counterexample with Bounded Perturbation}
\label{s6}
In this section an example is given to illustrate that $D_{\APS}$ will not always be Fredholm. There might be "infinite exchange" between the positive and the negative spectral subspace. This is possible, even if $A(t)$ has only discrete spectrum and its difference from $A(0)$ is bounded. The idea is to choose a perturbation of $A(0)$ such that the eigenspaces undergo a 90-degree-rotation "up to evolution with A(0)" under the evolution operator. The first step is to show that such an exchange works in a two dimensional subspace, with some bounds on the derivative of the perturbation. These bounds will then allow us to pass to an infinite direct sum, in which all eigenspaces are interchanged. This means that $\Ker(D_{\APS})=H_{<0}(0)\cap Q^{-1}H_{\geq 0}(T)$ will be infinite dimensional, whence $D_{\APS}$ is not Fredholm.

\subsection{Interchanging two Eigenspaces}
\begin{thm}
\label{swap}
There is $c\in \R$, such that the following is true: For any
$$a=\mat{\lambda_1}{0}{0}{\lambda_2},$$
with $\lambda_1,\lambda_2\in\R$, there is a $C^2$-family $(b(t))_{t\in[0,1]}$ of self-adjoint operators on $\C^2$ such that for $\lambda:=|\lambda_1-\lambda_2|+1$ we have
\begin{itemize}
\item$||b(t)||\leq 2$\\
\item$||b'(t)||\leq c \lambda$\\
\item$b(0)=b(1)=0$\\
\item$q(1,0)e_1\in \spann(e_2)$,
\end{itemize}
where q is the evolution operator associated with $a+b(t)$ and $e_i$ denotes the $i-th$ standard unit vector.
\end{thm}
\head{Remark:}
The idea of the proof is to choose $b$ in such a way, that
$$ib(t)q(t,0)e_1=\gamma'(t)$$
"up to evolution with a", where $q$ is the evolution operator associated to the family $a+b(t)$ and $\gamma$ is a path from $e_1$ to $e_2$. Then $q(t,0)e_1=\gamma(t)$, "up to evolution with a", which preserves $\spann(e_2)$. Heuristically, this works well, because $q$ only depends on earlier values of $b$. In case $\lambda_1=-\lambda_2$ (which would be sufficient for a counterexample), this argument can be made rigorous by quotienting out the orbits of $exp(ita)$ and applying the Banach fixed point theorem to show that the above equation has a solution on the quotient. However, this proof is rather long and yields a less general result, so a less conceptual proof will be used below.

\begin{proof}
Let $\phi:[0,1]\rightarrow [0,\frac{\pi}{2}]$ be a smooth function (chosen independently of the $\lambda_i$) with 
\begin{align*}
|\phi'(t)|&\leq 2\\
\phi(0)&=0\\
\phi(1)&=\frac{\pi}{2}\\
\phi'(0)&=\phi'(1)=0.
\end{align*}
Let
$$b(t):=\mat{0}{i\phi'(t)exp(i(\lambda_1-\lambda_2)t)}{-i\phi'(t)exp(i(\lambda_2-\lambda_1)t)}{0}.$$
$b(t)$ is self-adjoint and the evolution operator of $a+b(t)$ is given by
$$q(t,0):=\mat{exp(i\lambda_1t)cos(\phi(t))}{-exp(i\lambda_1t)sin(\phi(t))}{exp(i\lambda_2t)sin(\phi(t))}{exp(i\lambda_2t)cos(\phi(t))},$$
as it can easily be seen that
$q(0,0)=id$
and straightforward calculation of both sides shows
$$\ddt q(t,0)=i(a+b(t))q(t,0).$$
The estimates follow, as $\phi'$ is bounded by 2, and $\phi''$ is bounded as well, as the interval is compact. Thus
$$||b(t)||=|\phi'(t)||exp(\pm i(\lambda_1-\lambda_2)t)|\leq 2$$
and there is a constant $c$ such that 
$$\norm{b'(t)}\leq |\phi''(t)|+|\phi'(t)(\lambda_1-\lambda_2)|\leq c\lambda.$$
$b(0)=b(1)=0$ follows from $\phi'(0)=\phi'(1)=0$.
Finally, note that
$$q(1,0)e_1=exp(i\lambda_2)e_2\in \spann(e_2).$$
\end{proof}

\subsection {Interchanging all Eigenspaces}
Now that we have seen how to interchange two eigenspaces, it remains to show that we can do it with infinitely many simultaneously.

\begin{thm}
Let $H:=\bigoplus\limits_{i=0}^\infty\C^2$ and let $(\lambda_i)_{i\geq0}$ be an unbounded increasing sequence of positive real numbers.
Let
$$A_0:=\bigoplus\limits_{i=0}^\infty a_i$$
with 
$$a_i=\mat{-\lambda_i}{0}{0}{\lambda_i}.$$
There is a bounded family $B:[0,T]\rightarrow B(H)$ such that $A(t):=A_0+B(t)$ satisfies the requirements from the setting and such that $D_{\APS}$ is not Fredholm.
\end{thm}
\head{Remark:} The form of $A_0$ was chosen for simplicity. The construction works with any self-adjoint operator that has pairs of arbitrarily large positive and negative eigenvalues in the same order of magnitude.
\begin{proof}
 For $i\geq0$, let $b_i$ and $q_i$ be chosen as $b$ and $q$ in theorem $\ref{swap}$ with $\lambda_1=-\lambda_i$ and $\lambda_2=\lambda_i$. Let 
$$B:=\bigoplus\limits_{i=0}^\infty b_i$$
$$Q:=\bigoplus\limits_{i=0}^\infty q_i,$$
where the direct sum is meant to include taking the closure. Define 
$$A(t):=A_0+B(t).$$
As every $b_i(t)$ is bounded by 2, $B(t)=A(t)-A(0)$ is bounded by 2. Thus 
$$\Dom(A(t))=\Dom(A_0)$$
with equivalent norms for all $t$. By the Kato-Rellich theorem, $A_0+B(t)$ is self-adjoint.\\\\

For $|\lambda_i|>4$, $a_i+b_i(t)$ is invertible with inverse bounded by the inequality
$$\norm{(a_i+b_i(t))^{-1}}\leq \frac{1}{\norm{a_i^{-1}}^{-1}-\norm{b_i(t)}}\leq \frac{2}{\lambda_i}$$
for every $t\in[0,1]$.
Thus 
$$\bigoplus\limits_{|\lambda_i|>4}(a_i+b_i(t))^{-1}$$
is a well defined inverse of
$$\bigoplus\limits_{|\lambda_i|>4}a_i+b_i(t).$$
As only finitely many $\lambda_i$ are smaller than 4 and operators between finite dimensional spaces are Fredholm, we can conclude that
$$A(t)=\bigoplus\limits_{|\lambda_i|\leq4}(a_i+b_i(t))\oplus \bigoplus\limits_{|\lambda_i|>4}(a_i+b_i(t))$$
is a direct sum of two Fredholm operators and thus Fredholm.\\\\

Let $\iota_i$ denote the inclusion of the $i$-th summand. 
For 
$$x=\insum{i} \iota_i(x_i)\in \Dom(A_0),$$ 
we know that
$$\insum{i} \ddt(A(t)\iota_i(x_i)) =\insum{i} \iota_i(b_i'(t)x_i)$$
converges absolutely, as $||b_i'||\leq 2c(\lambda_i+1)$ (with $c$ as  above). Thus $A(t)$ is strongly differentiable with strong derivative 
$$A'=\bigoplus\limits_{i=0}^\infty b_i',$$
which is uniformly bounded by $2c$ in $B(W,H)$.
As every $b_i'$ is continuous, $A'$ is strongly continuous on the linear span of all $\iota_i(\C^2)$. As this is a dense subspace of $W$, Lemma \ref{densecont} implies that $A'$ is strongly continuous in $B(W,H)$, hence $A$ is strongly continuously differentiable.\\\\

Let $\tilde q_i$ be the component of the evolution operator of $A$ on the i-th direct summand. As this satisfies  $\tilde q_i'(t,0)=i(a_i+b_i(t))\tilde q_i(t,0)$ and $\tilde q_i(0,0)=id$, we have $\tilde q_i=q_i$ by \ref{uniq}. Thus $Q$ agrees with the evolution operator associated with $A$ on every summand, so they are the same.

 For all $i\in\N$, $\iota_i(e_1)$ is a negative eigenvector of $A(0)=A_0$, but 
$$Q(1,0)\iota_i(e_1)=\iota_i(q_i(1,0)e_1)\in \spann(\iota_i(e_2))$$
is a positive eigenvector of $A(1)=A_0$ by construction. Thus 
$$\Ker(D_{\APS})\cong \Ker(\hat P_{<0}(T)_r)=  H_{<0}(0)\cap Q(0,1)H_{\geq 0}(1)=\overline{\{\spann\{\iota_i(e_1)|i\in\N\}}$$
is infinite dimensional and hence $D_{\APS}$ is not Fredholm.\\
\end{proof}

\head{Remark:} Note that requiring only $D_{\APS}$ to be Fredholm is not sufficient either. To see this extend the family constructed above with another one constructed similarly, such that there is still no spectral flow, $A(2)=A(0)$ and $Q(2,1)$ swaps $\iota_i(e_2)$ and $\iota_{i+1}(e_1)$ (up to a scalar) for all $i\in\N$. Then $Q(2,0)$ sends $\iota_i(e_1)$ to $\iota_{i+1}(e_1)$, $\iota_{i+1}(e_2)$ to $\iota_i(e_2)$ and $\iota_0(e_2)$ to $\iota_0(e_1)$. Thus
\begin{align*}
Q(0,2)H_{<0}(2)&=Q(2,0)^{-1}\overline{span(\{\iota_i(e_1)|i\in\N\})}\\
&=\overline{span(\{\iota_i(e_1)|i\in \N\}\cup\{\iota_0(e_2)\})}\\
&=H_{<0}(0)\oplus span(\iota_0(e_1)).
\end{align*}
This means $(P_{<0}(0),\hat P_{<0}(2))$ is a Fredholm pair with index -1, i.e. $D_{\APS}$ is Fredholm with index -1, but $\sfl(A)$ is 0. Thus Fredholmness really has to be required for every initial interval rather than just the whole interval.

\section*{Acknowledgments}
I want to thank Matthias Lesch and Koen van den Dungen for their support, guidance, and feedback, as well as for suggesting this interesting topic, which allowed me to be certain to have some results early on, while being able to extend them until the end.

I am also grateful to Christoph Brinkmann for talks that helped me structure my understanding of the topic and for reading through my (ever changing) final draft, improving both the readability and the correctness of this paper with his feedback.


\begin{thebibliography}{99}
\bibitem{AvSeSi}[AvSeSi]:J. Avron, R. Seiler, B. Simon \emph{The Index of a Pair of Projections} J. Funct. Anal. ,1994
\bibitem {BaSt}[B\"aSt]: C. B\"ar and A. Strohmaier,\emph{An Index Theorem for Lorentzian Spacetimes with Compact Spacelike Cauchy Boundary}, American Journal of Mathematics, Volume 141, Number 5, 2019
\bibitem{HiPh}[HiPh]: E. Hille and R.S. Phillips, \emph{Functional Analysis and Semi-Groups}, American Mathematical Society, 1957
\bibitem{KA}[KA]: T. Kato, \emph{Perturbation Theory for Linear Operators}, Springer-Verlag, Berlin/Heidelberg, 1995
\bibitem{LE}[LE]: M.Lesch, \emph{The Uniqueness of the Spectral Flow on Spaces of Unbounded Self-adjoint Fredholm Operators}, arXiv:math/0401411v2 ,2004 (Published in Spectral geometry of manifolds with boundary and decomposition of manifolds (B. Booss-Bavnbek and K. P. Wojciechowski Eds.), Contemp. Math. 366, Amer. Math. Soc., Providence, 2005)
\bibitem{PA}[PA]: A. Pazy, \emph{Semigroups of Linear Operators and Applications to Partial Differential Equations}, Springer-Verlag, New York, 1983
\bibitem{PE} [PE]: G. K. Pedersen, \emph{Analysis Now}, Springer-Verlag, New York, 1989
\bibitem {PH}[PH]:J. Phillips, \emph{Self-adjoint Fredholm operators and spectral flow} Canad. Math. Bull., 1996
\bibitem {RoSa}[RoSa]: J.Robbin and D. Salamon, \emph{the Specrtral flow and the Maslov index}, Bulletin of the London Mathematical Society, 1995
\end{thebibliography}
\end{document}